\documentclass[final]{siamltexmm}
\usepackage{mymacros}
\usepackage{appendix}
\usepackage[utf8]{inputenc}

\title{Matrix probing and its conditioning\thanks{Dated February 2011.}}
\author{Jiawei Chiu \thanks{Corresponding author. \email{jiawei@mit.edu}. Department of Mathematics, MIT, Cambridge, MA 02139, USA.}\and Laurent Demanet \thanks{Department of Mathematics, MIT, Cambridge, MA 02139, USA.}}

\begin{document}
\maketitle
\newcommand{\slugmaster}{%
}

\begin{abstract}
When a matrix $A$ with $n$ columns is known to be well approximated by a linear combination of basis matrices $B_1,\ldots,B_p$, we can apply $A$ to a random vector and solve a linear system to recover this linear combination. The same technique can be used to obtain an approximation to $A^{-1}$. A basic question is whether this linear system is well-conditioned. This is important for two reasons: a well-conditioned system means (1) we can invert it and (2) the error in the reconstruction can be controlled. In this paper, we show that if the Gram matrix of the $B_j$'s is sufficiently well-conditioned and each $B_j$ has a high numerical rank, then $n\propto p \log^2 n$ will ensure that the linear system is well-conditioned with high probability. Our main application is probing linear operators with smooth pseudodifferential symbols such as the wave equation Hessian in seismic imaging \cite{demanet2011matrix}. We also demonstrate numerically that matrix probing can produce good preconditioners for inverting elliptic operators in variable media.
\end{abstract}

\section*{\small Acknowledgments}
{\footnotesize JC is supported by the A*STAR fellowship from Singapore. LD is supported in part by a grant from the National Science Foundation and the Alfred P. Sloan foundation. We thank Holger Rauhut for the interesting discussions, and the anonymous referees for suggesting several improvements to the paper.}

\pagestyle{myheadings}
\thispagestyle{plain}

\section{Introduction}
The earliest randomized algorithms include Monte Carlo integration and Monte Carlo Markov chains \cite{andrieu03}. These are standard techniques in numerical computing with widespread applications from physics, econometrics to machine learning. However, they are often seen as the methods of last resort, because they are easy to implement but produce solutions of uncertain accuracy.

In the last few decades, a new breed of randomized algorithms has been developed by the computer science community. These algorithms remain easy to implement, and in addition, have failure probabilities that are provably negligible. In other words, we have rigorous theory to ensure that these algorithms perform consistently well. Moreover, their time complexity can be as good as the most sophisticated deterministic algorithms, e.g., Karp-Rabin's pattern matching algorithm \cite{karp1987efficient} and Karger's min-cut algorithm \cite{karger96}.

In recent years, equally attractive randomized algorithms are being developed in the numerical community. For example, in compressed sensing \cite{candes2006stable}, we can recover sparse vectors with random measurement matrices and $\ell^1$ minimization. Another interesting example is that we can build a good low rank approximation of a matrix by \emph{applying it to random vectors} \cite{halko}.

Our work carries a similar flavor: often, the matrix $A$ can be approximated as a linear combination of a small number of matrices and the idea is to obtain these coefficients by applying $A$ to a random vector or just a few of them. We call this ``forward matrix probing.'' What is even more interesting is that we can also probe for $A^{-1}$ by applying $A$ to a random vector. We call this ``backward matrix probing'' for a reason that will be clear in Section \ref{sec:probeinv}.

Due to approximation errors, the output of ``backward probing'' denoted as $C$, is only an approximate inverse. Nevertheless, as we will see in Section \ref{sec:numerical}, $C$ serves very well as a preconditioner for inverting $A$, and we believe that its performance could match that of multigrid methods for elliptic operators in smooth media.

We like to add that the idea of ``matrix probing'' is not new. For example, Chan \cite{chan1985survey, chan1992interface} et. al. use the technique to approximate $A$ with a sparse matrix. Another example is the work by Pfander et. al. \cite{pfander2008identification} where the same idea is used in a way typical in compressed sensing. In the next section, we will see that their set-up is fundamentally different from ours.

\subsection{Forward matrix probing}\label{sec:basic}

Let $\mathcal{B}=\{B_1,\ldots,B_p\}$ where each $B_j \in \cplexes^{m \times n}$ is called a basis matrix.  Note that $\mathcal{B}$ is specified in advance. Let $u$ be a Gaussian or a Rademacher sequence, that is each component of $u$ is independent and is either a standard normal variable or $\pm 1$ with equal probability.

Define the matrix $L\in \cplexes^{m\times p}$ such that its $j$-th column is $B_j u$. Let $A\in \cplexes^{m\times n}$ be the matrix we want to probe and suppose $A$ lies in the span of $\mathcal{B}$. Say
$$A=\sum_{i=1}^p c_i B_i \mbox{ for some } c_1,\ldots,c_p\in \cplexes.$$

Observe that $Au = \sum_{i=1}^p c_i (B_i u)=Lc$. Given the vector $Au$, we can obtain the coefficient vector $c=(c_1,\ldots,c_p)^T$ by solving the linear system
\begin{equation}\label{eq:tosolve}
Lc=Au.
\end{equation}

In practice, $A$ is not exactly in the span of a small $\mathcal{B}$ and Equation (\ref{eq:tosolve}) has to be solved in a least squares sense, that is $c=L^{+}(Au)$ where $L^+$ is the pseudoinverse of $L$.

We will assume that $p\leq n$. Otherwise there are more unknowns than equations and there is no unique solution if there is any. This differs from the set-up in \cite{pfander2008identification} where $n\gg p$ but $A$ is assumed to be a sparse linear combination of $B_1,\ldots,B_p$.

\subsection{Conditioning of $L$}\label{sec:conditioningprelude}

Whether Equation (\ref{eq:tosolve}) can be solved accurately depends on $\cond(L)$, the condition number of $L$. This is the ratio between the largest and the smallest singular values of $L$ and can be understood as how different $L$ can stretch or shrink a vector. 

Intuitively, whether $\cond(L)$ is small depends on the following two properties of $\mathcal{B}$.

\begin{enumerate}
\item The $B_i$'s ``act differently'' in the sense that $\inner{B_j}{B_k}\simeq \delta_{jk}$ for any $1\leq j,k \leq p$.\footnote{Note that $\inner{\cdot}{\cdot}$ is the Frobenius inner product and $\delta_{jk}$ is the Kronecker delta.}
\item Each $B_i$ has a high rank so that $B_1 u,\ldots,B_p u \in \cplexes^{n}$ exist in a high dimensional space.
\end{enumerate}

When $\mathcal{B}$ possesses these two properties and $p$ is sufficiently small compared to $n$, it makes sense that $L$'s columns, $B_1 u,\ldots, B_p u$, are likely to be independent, thus guaranteeing that $L$ is invertible, at least.

We now make the above two properties more precise. Let
\begin{equation}\label{eq:MN}
M=L^* L \in \cplexes^{p \times p} \mbox{ and } N=\E M.
\end{equation}

Clearly, $\cond(M)= \cond(L)^2$. If $\E M$ is ill-conditioned, there is little chance that $M$ or $L$ is well-conditioned. This can be related to Property 1 by observing that
\begin{equation}\label{eq:Njk}
N_{jk} = \E M_{jk} =\tr({B_j}^* B_k) = \inner{B_j}{ B_k}.
\end{equation}

If $\inner{B_j}{B_k} \simeq \delta_{jk}$, then the Gram matrix $N$ is approximately the identity matrix which is well-conditioned. Hence, a more quantitative way of putting Property 1 is that we have control over $\kappa(B)$ defined as follows.
\begin{definition}\label{def:kappa}
Let $\mathcal{B} = \{B_1,\ldots,B_p\}$ be a set of matrices. Define its condition number $\kappa(\mathcal{B})$ as the condition number of the matrix $N\in \cplexes^{p \times p}$ where $N_{jk}=\inner{B_j}{B_k}$.
\end{definition}

On the other hand, Property 2 can be made precise by saying that we have control over $\lambda(\mathcal{B})$ as defined below.

\begin{definition}\label{def:weakcondition}
Let $A\in \cplexes^{m\times n}$. Define its weak condition number\footnote{Throughout the paper, $\norm{\cdot}$ and $\norm{\cdot}_F$ denote the spectral and Frobenius norms respectively.} as
$$\lambda(A) = \frac{\norm{A} n^{1/2}}{\norm{A}_F}.$$

Let $\mathcal{B}$ be a set of matrices. Define its (uniform) weak condition number as $$\lambda(\mathcal{B})= \max_{A\in \mathcal{B}} \lambda(A).$$
\end{definition}

We justify the nomenclature as follows. Suppose $A\in \cplexes^{n\times n}$ has condition number $k$, then $\norm{A}_F^2=\sum_{i=1}^{n} \sigma_i^2 \geq n\sigma_{\min}^2 \geq n \norm{A}^2/k^2$. Taking square root, we obtain $\lambda(A) \leq k$. In other words, any well-conditioned matrix is also weakly well-conditioned. And like the usual condition number, $\lambda(\mathcal{A})\geq 1$ because we always have $\norm{A}_F \leq n^{1/2}\norm{A}$.

The numerical rank of a matrix $A$ is $\norm{A}_F^2/\norm{A}^2=n \lambda(A)^{-2}$, thus having a small $\lambda(A)$ is the same as having a high numerical rank. We also want to caution the reader that $\lambda(\mathcal{B})$ is defined very differently from $\kappa(\mathcal{B})$ and is not a weaker version of $\kappa(\mathcal{B})$.

Using classical concentration inequalties, it was shown \cite{demanet2011matrix} that when $\lambda(\mathcal{B})$ and $\kappa(\mathcal{B})$ are fixed, $p=\tilde{O}({n}^{1/2})$\footnote{Note that $\tilde{O}(n)$ denotes $O(n \log^c n)$ for some $c>0$. In other words, ignore log factors.} will ensure that $L$ is well-conditioned with high probability.

In this paper, we establish a stronger result, namely that $p=\tilde{O}(n)$ suffices. The implication is that we can expect to recover $\tilde{O}(n)$ instead of $\tilde{O}(n^{1/2})$ coefficients. The exact statement is presented below.

\begin{theorem}[Main result]\label{thm:main}
Let $C_1,C_2>0$ be numbers given by Remark \ref{remark:C1C2} in the Appendix.
Let $\mathcal{B}=\{B_1,\ldots,B_p\}$ where each $B_j\in \cplexes^{m\times n}$. Define $L\in \cplexes^{n\times p}$ such that its $j$-th column is $B_j u$ where $u$ is either a Gaussian or Rademacher sequence. Let $M=L^* L$, $N=\E M$ $\kappa=\kappa(\mathcal{B})$ and $\lambda = \lambda(\mathcal{B})$. Suppose $$n\geq p\left(C \kappa  \lambda \log n\right)^2 \mbox{ for some }C \geq 1.$$
Then
$$\P{\norm{M-N}\geq \frac{t\norm{N}}{\kappa}}\leq 2 C_2 p n^{1-\alpha} \mbox{ where } \alpha=\frac{t C}{e C_1}.$$
\end{theorem}

The number $C_1$ is small. $C_2$ may be large but it poses no problem because $n^{-\alpha}$ decays very fast with larger $n$ and $C$. With $t=1/2$, we deduce that with high probability,
$$\cond(M) \leq 2\kappa+1.$$

In general, we let $0<t<1$ and for the probability bound to be useful, we need $\alpha>2$, which implies $C>2e C_1>1$. Therefore the assumption that $C\geq 1$ in the theorem can be considered redundant.

We remark that Rauhut and Tropp have a new result (a Bernstein-like tail bound) that may be used to refine the theorem. This will be briefly discussed in Section \ref{sec:1Dstatstudy} where we conduct a numerical experiment.

Note that when $u$ is a Gaussian sequence, $M$ resembles a Wishart matrix for which the distribution of the smallest eigenvalue is well-studied \cite{edelman1989eigenvalues}. However, each row of $L$ is not independent, so results from random matrix theory cannot be used in this way.

An intermediate result in the proof of Theorem \ref{thm:main} is the following. It conveys the essence of Theorem \ref{thm:main} and may be easier to remember.
\begin{theorem}\label{thm:mainweak}
Assume the same set-up as in Theorem \ref{thm:main}. Suppose $n=\tilde{O}(p)$. Then
$$\E \norm{M-N} \leq C (\log n) \norm{N}(p/n)^{1/2} \lambda \mbox{ for some }C>0.$$
\end{theorem}

A numerical experiment in Section \ref{sec:1Dstatstudy} suggests that the relationship between $p$ and $n$ is not tight in the $\log$ factor. Our experiment show that for $\E \norm{M-N}/\norm{N}$ to vanish as $p\rightarrow \infty$, $n$ just needs to increase faster than $p \log (np)$, whereas Theorem \ref{thm:mainweak} requires $n$ to grow faster than $p \log^2 n$.

Next, we see that when $L$ is well-conditioned, the error in the reconstruction is also small.
\begin{proposition}\label{thm:accuracy}
Assume the same set-up as in Theorem \ref{thm:main}. Suppose $A=  \sum_{j=1}^p d_j B_j + E$ where $\norm{E}\leq \eps$ and assume whp,
\begin{equation*}\label{eq:accuracyassume}
\norm{M-N}\leq \frac{t\norm{N}}{\kappa} \mbox{ for some }0<t<1.
\end{equation*}
Let $c=L^+ A u$ be the recovered coefficients. Then whp,
$$\norm{A-\sum_{j=1}^p c_j B_j}\leq O\left(\eps \lambda \left(\frac{\kappa p}{1-t}\right)^{1/2}\right).$$
\end{proposition}

If $\eps=o(p^{-1/2})$, then the proposition guarantees that the overall error goes to zero as $p\rightarrow \infty$. Of course, a larger $n$ and more computational effort are required.

\subsection{Multiple probes}\label{sec:multipleprobes}
Fix $n$ and suppose $p>n$. $L$ is not going to be well-conditioned or even invertible. One way around this is to probe $A$ with multiple random vectors $u_1,\ldots,u_q\in \cplexes^n$ at one go, that is to solve
$$L' c = A' u,$$
where the $j$-th column of $L'$ and $A'u$ are respectively
$$\left(\begin{array}{c}
B_j u_1 \\
\vdots\\
 B_j u_q
\end{array}\right) \mbox{ and }
\left(\begin{array}{c}
A u_1 \\
\vdots\\
 A u_q
\end{array}\right) 
.$$

For this to make sense, $A'=I_{q} \otimes A$ where $I_{q}$ is the identity matrix of size $q$. Also define $B'_j = I_{q} \otimes B_j$ and treat the above as probing $A'$ assuming that it lies in the span of $\mathcal{B}'=\{B'_1,\ldots,B'_p\}$.

Regarding the conditioning of $L'$, we can apply Theorem \ref{thm:main} to $A'$ and $\mathcal{B}'$. It is an easy exercise (cf. Proposition \ref{thm:condunchanged}) to see that the condition numbers are unchanged, that is $\kappa(\mathcal{B})=\kappa(\mathcal{B}')$ and $\lambda(\mathcal{B})=\lambda(\mathcal{B}')$. Applying Theorem \ref{thm:main} to $A'$ and $\mathcal{B}'$, we deduce that $\cond(L)\leq 2\kappa+1$ with high probability provided that
$$n q \propto p( \kappa \lambda \log n)^2.$$

Remember that $A$ has only $mn$ degrees of freedom; while we can increase $q$ as much as we like to improve the conditioning of $L$, the problem set-up does not allow $p>mn$ coefficients. In general, when $A$ has rank $\tilde{n}$, its degrees of freedom is $\tilde{n}(m+n-\tilde{n})$ by considering its SVD.

\subsection{When to probe}\label{sec:motivation}
Matrix probing is an especially useful technique when the following holds.

\begin{enumerate}
\item We know that the probed matrix $A$ can be approximated by a small number of basis matrices that are specified in advance. This holds for operators with smooth pseudodifferential symbols, which will be studied in Section \ref{sec:pdos}.
\item Each matrix $B_i$ can be applied to a vector in $\tilde{O}(\max(m,n))$  time using only $\tilde{O}(\max(m,n))$ memory.
\end{enumerate}

The second condition confers two benefits. First, the coefficients $c$ can be recovered fast, assuming that $u$ and $Au$ are already provided. This is because $L$ can be computed in $\tilde{O}(\max(m,n)p)$ time and Equation (\ref{eq:tosolve}) can be solved in $O(m p^2+p^3)$ time by QR factorization or other methods. In the case where increasing $m,n$ does not require a bigger $\mathcal{B}$ to approximate $A$, $p$ can be treated as a constant and the recovery of $c$ takes only $\tilde{O}(\max(m,n))$ time.

Second, given the coefficient vector $c$, $A$ can be applied to any vector $v$ by summing over $B_i v$'s in $\tilde{O}(\max(m,n)p)$ time . This speeds up iterative methods such as GMRES and Arnoldi.

\subsection{Backward matrix probing}\label{sec:probeinv}

A compelling application of matrix probing is computing the pseudoinverse $A^{+}$ of a matrix $A\in \cplexes^{m\times n}$ when $A^{+}$ is known to be well-approximated in the space of some $\mathcal{B}=\{B_1,\ldots,B_p\}$. This time, we probe $A^{+}$ by applying it to a random vector $v=Au$ where $u$ is a Gaussian or Rademacher sequence that we generate.

Like in Section \ref{sec:basic}, define $L\in \cplexes^{n\times p}$ such that its $j$-th column is $B_j v = B_j A u$. Suppose $A^{+} = \sum_{i=1}^p c_i B_i \mbox{ for some } c_1,\ldots, c_p \in \cplexes$. Then the coefficient vector $c$ can be obtained by solving
\begin{equation}\label{eq:tosolve2}
Lc = A^{+} v = A^{+} A u.
\end{equation}

The right hand side is $u$ projected onto $\nullspace(A)^{\perp}$ where $\nullspace(A)$ is the nullspace of $A$. When $A$ is invertible, $A^{+} A u$ is simply $u$. We call this ``backward matrix probing'' because the generated random vector $u$ appears on the opposite side of the matrix being probed in Equation (\ref{eq:tosolve2}). The equation suggests the following framework for probing $A^{+}$.

\begin{algorithm}[Backward matrix probing]\label{alg:probeinv}
Suppose $A^{+}=\sum_{i=1}^p c_i B_i$. The goal is to retrieve the coefficients $c_1,\ldots,c_p$.
\begin{enumerate}
\item Generate $u\sim N(0,1)^n$ iid.
\item Compute $v=Au$.
\item Filter away $u$'s components in $\nullspace(A)$. Call this $\tilde{u}$.
\item Compute $L$ by setting its $j$-column to $B_j v$.
\item Solve for $c$ the system $Lc=\tilde{u}$ in a least squares sense.
\end{enumerate}
\end{algorithm}

In order to perform the filtering in Step 3 efficiently, prior knowledge of $A$ may be needed. For example, if $A$ is the Laplacian with periodic boundary conditions, its nullspace is the set of constant functions and Step 3 amounts to subtracting the mean from $u$. A more involved example can be found in \cite{demanet2011matrix}. In this paper, we invert the wave equation Hessian, and Step 3 entails building an illumination mask. Further comments on \cite{demanet2011matrix} are located in Section \ref{sec:hessian} of this paper.

For the conditioning of $L$, we may apply Theorem \ref{thm:main} with $\mathcal{B}$ replaced with $\mathcal{B}_A:=\{B_1 A,\ldots,B_p A\}$ since the $j$-th column of $L$ is now $B_j A u$. Of course, $\kappa(\mathcal{B}_A)$ and $\lambda(\mathcal{B}_A)$ can be very different from $\kappa(\mathcal{B})$ and $\lambda(\mathcal{B})$; in fact, $\kappa({\mathcal{B}}_A)$ and $\lambda({\mathcal{B}}_A)$ seem much harder to control because it depends on $A$. Fortunately, as we shall see in Section \ref{sec:order}, knowing the ``order'' of $A^{+}$ as a pseudodifferential operator helps in keeping these condition numbers small.

When $A$ has a high dimensional nullspace but has comparable nonzero singular values, $\lambda(\mathcal{B}_A)$ may be much larger than is necessary. By a change of basis, we can obtain the following tighter result.
\begin{corollary}\label{thm:main2}
Let $C_1,C_2>0$ be numbers given by Remark \ref{remark:C1C2} in the Appendix.
Let $A\in \cplexes^{m\times n}$, $\tilde{n}=\rank(A)$ and $\mathcal{B}_A=\{B_1 A,\ldots,B_p A\}$ where each $B_j\in \cplexes^{n\times m}$. Define $L\in \cplexes^{n\times p}$ such that its $j$-th column is $B_j A u$ where $u\sim N(0,1)^n$ iid. Let $M=L^* L$, $N=\E M$, $\kappa=\kappa(\mathcal{B}_A)$ and $\lambda = (\tilde{n}/n)^{1/2}\lambda(\mathcal{B}_A)$. Suppose
$$\tilde{n}\geq p\left(C \kappa  \lambda \log \tilde{n}\right)^2 \mbox{ for some }C \geq 1.$$
Then
$$\P{\norm{M-N}\geq \frac{t\norm{N}}{\kappa}}\leq (2 C_2 p) \tilde{n}^{1-\alpha} \mbox{ where } \alpha=\frac{t C}{e C_1}.$$
\end{corollary}

Notice that $\tilde{n}=\rank(A)$ has taken the role of $n$, and our $\lambda$ is now $\max_{1\leq j\leq p}\frac{\tilde{n}^{1/2} \norm{B_j A}}{\norm{B_j A}_F}$, which ignores the $n-\tilde{n}$ zero singular values of each $B_j A$ and can be much smaller than $\lambda(\mathcal{B}_A)$.

\section{Proofs}\label{sec:proof}

\subsection{Proof of Theorem \ref{thm:main}}
Our proof is decoupled into two components: one linear algebraic and one probabilistic. The plan is to collect all the results that are linear algebraic, deterministic in nature, then appeal to a probabilistic result developed in the Appendix.

To facilitate the exposition, we use a different notation for this section. We use lower case letters as \emph{superscripts} that run from 1 to $p$ and Greek symbols as subscripts that run from 1 to $n$ or $m$. For example, the set of basis matrices is now $\mathcal{B}=\{B^1,\ldots,B^p\}$.

Our linear algebraic results concern the following variables.
\begin{enumerate}
\item Let $T^{jk} = {B^{j}}^* B^k\in \cplexes^{n \times n}$ and $T_{\xi \eta}\in \cplexes^{p \times p}$ such that the $(j,k)$-th entry of $T_{\xi \eta}$ is the $(\xi,\eta)$-th entry of $T^{jk}$.
\item Let $Q=\sum_{1\leq \xi ,\eta\leq n} T_{\xi \eta}^* T_{\xi \eta}$.
\item Let $S=\sum_{j=1}^p B^j {B^j}^* \in \cplexes^{m \times m}$.
\item Let $F$ and $G$ be block matrices $(T_{\xi \eta})_{1\leq \xi,\eta\leq n}$ and $(T_{\xi \eta}^*)_{1\leq \xi,\eta\leq n}$ respectively.
\end{enumerate}

The reason for introducing $T$ is that $M$ can be written as a quadratic form in $T_{\xi \eta}$ with input $u$:
\begin{equation*}
M=\sum_{1\leq \xi, \eta\leq n} u_{\xi} u_{\eta} T_{\xi \eta}.
\end{equation*}

Since $u_{\xi}$ has unit variance and zero mean, $N=\E M = \sum_{\xi=1}^n T_{\xi \xi}$.

Probabilistic inequalties applied to $M$ will involve $T_{\xi \eta}$, which must be related to $\mathcal{B}$. The connection between these $n$ by $n$ matrices and $p$ by $p$ matrices lies in the identity
\begin{equation}\label{eq:Telement}
T_{\xi \eta}^{jk} = \sum_{\zeta=1}^m \overline{B^j_{\zeta \xi}} B^k_{\zeta \eta}.
\end{equation}

The linear algebraic results are contained in the following propositions.
\begin{proposition}\label{thm:Txixipositive}
For any $1\leq \xi,\eta \leq n$,$$T_{\xi \eta}=T_{\eta \xi}^*.$$
Hence, $T_{\xi \xi},N$ are all Hermitian. Moreover, they are positive semidefinite.
\end{proposition}
\begin{proof}
Showing that $T_{\xi \eta} = T_{\eta \xi}^*$ is straightforward from Equation (\ref{eq:Telement}). We now check that $T_{\xi \xi}$ is positive semidefinite. Let $v\in \cplexes^p$. By Equation (\ref{eq:Telement}), $v^* T_{\xi \xi} v=\sum_{\zeta} \sum_{jk} \overline{v^j} v^k \overline{B^j_{\zeta \xi}} B^k_{\zeta \xi}=\sum_{\zeta}\abs{\sum_k v^k B^k_{\zeta \xi}}^2 \geq 0$. It follows that $N=\sum_{\xi}T_{\xi\xi}$ is also positive semidefinite.
\end{proof}

\begin{proposition}\label{thm:QBSB}
$$Q^{jk} = \tr({B^j}^* S B^k) \mbox{ and } Q=\sum_{1\leq \xi, \eta \leq n} T_{\xi \eta} T_{\xi \eta}^* .$$
\end{proposition}
\begin{proof}
By Equation (\ref{eq:Telement}), $Q^{jk} =\sum_{l} \inner{T^{lj}}{T^{lk}}=\sum_{l} \tr({B^j}^* B^l {B^l}^* B^k)$. The summation and trace commute to give us the first identity. Similarly, the $(j,k)$-th entry of $\sum_{\xi \eta} T_{\xi \eta} T_{\xi \eta}^*$ is $\sum_{l} \inner{T^{kl}}{T^{jl}}=\sum_{l} \tr({B^l}^* B^k {B^j}^* B^l)$. Cycle the terms in the trace to obtain $Q^{jk}$.
\end{proof}

\begin{proposition}\label{thm:Nestimate}
Let $u\in \cplexes^p$ be a unit vector. Define $U=\sum_{k=1}^p u^k B^k \in \cplexes^{m \times n}$. Then
$$\norm{U}_F^2 \leq \norm{N}.$$
\end{proposition}
\begin{proof}
$\norm{U}_F^2 = \tr(U^* U) = \tr(\sum_{jk} \overline{u^j} u^k {B^j}^* B^k)$. The sum and trace commute and due to Equation (\ref{eq:Njk}), $\norm{U}_F^2=\sum_{jk} \overline{u^j} u^k N^{jk}\leq \norm{N}$.
\end{proof}

\begin{proposition}\label{thm:QSN}
$$\norm{Q} \leq \norm{S} \norm{N}.$$
\end{proposition}
\begin{proof}
$Q$ is Hermitian, so $\norm{Q}=\max_u u^* Q u$ where $u\in \cplexes^p$ has unit norm. Now let $u$ be an arbitrary unit vector and define $U=\sum_{k=1}^p u^k B^k$. By Proposition \ref{thm:QBSB}, $u^* Q u = \sum_{jk}\overline{u^j} u^k Q^{jk}=\tr(\sum_{jk} \overline{u^j} u^k {B^j}^* S B^k)=\tr(U^* S U)$. Since $S$ is positive definite, it follows from ``$\norm{AB}_F\leq \norm{A}\norm{B}_F$'' that $u^* Q u = \norm{S^{1/2} U}_F^2\leq \norm{S} \norm{U}_F^2$. By Proposition \ref{thm:Nestimate}, $u^* Q u \leq \norm{S} \norm{N}$.
\end{proof}

\begin{proposition}\label{thm:BN}
For any $1\leq j\leq p$,
$$\norm{B^j} \leq \lambda n^{-1/2} \norm{N}^{1/2}.$$
It follows that
$$\norm{Q}=\norm{\sum_{\xi \eta} T_{\xi \eta} T_{\xi \eta}^*}\leq p\lambda^2 \norm{N}^2/n.$$
\end{proposition}
\begin{proof}
We begin by noting that $\norm{N}\geq \max_{j} |N^{jj}| =\max_{j} \inner{B^j}{B^j} = \max_j \norm{B^j}_F^2$. From Definition \ref{def:weakcondition}, $\norm{B^j}\leq \lambda n^{-1/2}\norm{B^j}_F \leq \lambda n^{-1/2}\norm{N}^{1/2}$ for any $1\leq j \leq p$, which is our first inequality. It follows that $\norm{S}\leq \sum_{j=1}^p \norm{B^j}^2\leq p \lambda^2 \norm{N}/n$. Apply Propositions \ref{thm:QSN} and \ref{thm:QBSB} to obtain the second inequality.
\end{proof}

\begin{proposition}\label{thm:FN}
$F,G$ are Hermitian, and
$$\max(\norm{F},\norm{G})\leq \lambda^2 \norm{N}(p/n).$$
\end{proposition}
\begin{proof}
That $F,G$ are Hermitian follow from Proposition \ref{thm:Txixipositive}. Define $F'=(T^{jk})$ another block matrix. Since reindexing the rows and columns of ${F}$ does not change its norm, $\norm{F}=\norm{F'}$. By Proposition \ref{thm:BN}, $\norm{F'}^2\leq \sum_{j,k=1}^p \norm{T^{jk}}^2 \leq \sum_{j,k=1}^p \norm{B^j}^2\norm{B^k}^2 \leq \lambda^4 \norm{N}^2(p/n)^2$. The same argument works for $G$.
\end{proof}

We now combine the above linear algebraic results with a probabilistic result in Appendix \ref{sec:prob}. Prepare to apply Proposition \ref{thm:tailbound} with $A_{ij}$ replaced with $T_{\xi \eta}$. Note that $R=\sum_{\xi \eta} T_{\xi \eta} T_{\xi \eta}^*=Q$ by Proposition \ref{thm:QBSB}. Bound $\sigma$ using Propositions \ref{thm:BN} and  \ref{thm:FN}:
\begin{align*}\sigma &= C_1 \max(\norm{Q}^{1/2},\norm{R}^{1/2},\norm{F},\norm{G})\\
&\leq C_1 \norm{N} \max((p/n)^{1/2} \lambda, (p/n)\lambda^2)\\
&\leq C_1 \norm{N} (p/n)^{1/2} \lambda.
\end{align*}
The last step goes through because our assumption on $n$ guarantees that $(p/n)^{1/2}\lambda\leq 1$. Finally, apply Proposition \ref{thm:tailbound} with $t\norm{N}/\kappa = e \sigma u$. The proof is complete.

\subsection{Sketch of the proof for Theorem \ref{thm:mainweak}}
Follow the proof of Proposition \ref{thm:tailbound}. Letting $s=\log n$, we obtain
\begin{align*}
\E\norm{M-N}&\leq \left( \E \norm{M-N}^{s} \right)^{1/s} \\
&\leq C_1 (2C_2 n p)^{1/s} s \max(\norm{Q}^{1/2}, \norm{R}^{1/2},\norm{F},\norm{G})\\
& \leq C (\log n) \norm{N}(p/n)^{1/2} \lambda.
\end{align*}

\subsection{Proof of Proposition \ref{thm:accuracy}}
Recall that $A$ is approximately the linear combination $\sum_{j=1}^p d^j B^j$, while $\sum_{j=1}^p c^j B^j$ is the recovered linear combination. We shall first show that the recovered coefficients $c$ is close to $d$:
\begin{align*}
\norm{d-c}& = \norm{ L^+ Au - c}\\
&= \norm{L^+ (Lc+ Eu) - c}\\
&=\norm{L^+ E u}\\
& \leq \eps \norm{u} \left( \frac{\kappa}{(1-t)\norm{N}}\right)^{1/2}.
\end{align*}
Let $v$ be a unit $n$-vector. Let $L'$ be a $n\times p$ matrix such that its $j$-th column is $B^j v$. Now,
$$Av-\sum_{j=1}^p c^j B^j v =  (L' d + Ev) - L' c= Ev + L'(d-c).$$
Combining the two equations, we have
\begin{equation}\label{eq:tmp1}
\norm{A-\sum_{j=1}^p c^j B^j}\leq \eps + \eps \norm{L'}\norm{u} \left( \frac{\kappa}{(1-t)\norm{N}}\right)^{1/2}.
\end{equation}
With overwhelming probability, $\norm{u}=O(\sqrt{n})$. The only term left that needs to be bounded is $\norm{L'}$. This turns out to be easy because $\norm{B^j}\leq \lambda n^{-1/2}\norm{N}^{1/2}$ by Proposition \ref{thm:BN} and
$$\norm{L'}^2 \leq \sum_{j=1}^p \norm{B^j v}^2 \leq \lambda^2 \norm{N} p/n.$$

Substitute this into Equation (\ref{eq:tmp1}) to finish the proof.

\subsection{Proof of Corollary \ref{thm:main2}}\label{sec:proof3}
Let $u\sim N(0,1)^n$ iid. Say $A$ has a singular value decomposition $E\Lambda F^*$ where $\Lambda$ is diagonal. Do a change of basis by letting $u'=F^* u\sim N(0,1)^n$ iid, $B'_j = F^* B_j E$ and $\mathcal{B}'_{\Lambda}=\{B'_1 \Lambda,\ldots,B'_p \Lambda\}$. Equation (\ref{eq:tosolve}) is reduced to $L' c = \Lambda u'$ where the $j$-th column of $L'$ is $B'_j \Lambda u'$.

Since Frobenius inner products, $\norm{\cdot}$ and $\norm{\cdot}_F$ are all preserved under unitary transformations, it is clear that $\kappa(\mathcal{B}'_{\Lambda})=\kappa(\mathcal{B}_A)$ and $\lambda(\mathcal{B}'_{\Lambda}) = \lambda(\mathcal{B}_A)$. Essentially, for our purpose here, we may pretend that $A=\Lambda$.

Let $\tilde{n}=\rank(A)$. If $A$ has a large nullspace, i.e., $\tilde{n}\ll \min(m,n)$, then $B'_j \Lambda$ has $n-\tilde{n}$ columns of zeros and many components of $u'$ are never transmitted to the $B'_j$'s anyway. We may therefore truncate the length of $u'$ to $\tilde{n}$, let $\tilde{B}_j\in \cplexes^{n \times \tilde{n}}$ be $B'_j\Lambda$ with its columns of zeros chopped away and apply Theorem \ref{thm:main} with $\mathcal{B}$ replaced with $\tilde{\mathcal{B}}:=\{\tilde{B}_1,\ldots,\tilde{B}_p\}$. Observe that $\kappa(\tilde{\mathcal{B}})=\kappa(\mathcal{B}'_{\Lambda})$, whereas $\lambda(\tilde{\mathcal{B}})=(\tilde{n}/n)^{1/2} \lambda(\mathcal{B}'_{\Lambda})$ because $\norm{\tilde{B}_j}_F = \norm{B'_j \Lambda}_F$ and $\norm{\tilde{B}_j}=\norm{B'_j \Lambda}$ but $\tilde{B}_j$ has $\tilde{n}$ instead of $n$ columns. The proof is complete.

\section{Probing operators with smooth symbols}\label{sec:pdos}

\subsection{Basics and assumptions}\label{sec:pdobasics}

We begin by defining what a pseudodifferential symbol is.

\begin{definition}
Every linear operator $A$ is associated with a pseudodifferential symbol  $a(x,\xi)$ such that for any $u:\reals^d \rightarrow \reals$,
\begin{equation}\label{eq:symbolrep}
Au(x) = \int_{\xi\in \reals^d} e^{2\pi i \xi \cdot x} a(x,\xi) \hat{u}(\xi) d\xi
\end{equation}
where $\hat{u}$ is the Fourier transform of $u$, that is $\hat{u}(\xi) = \int_{x\in \reals^d} u(x) e^{-2\pi i \xi \cdot x} dx$.
\end{definition}

We refrain from calling $A$ a ``pseudodifferential operator'' at this point because its symbol has to satisfy some additional constraints that will be covered in Section \ref{sec:order}. What is worth noting here is the Schwartz kernel theorem which shows that every linear operator $A:\mathcal{S}(\reals^d) \rightarrow \mathcal{S}'(\reals^d)$ has a symbol representation as in Equation (\ref{eq:symbolrep}) and in that integral, $a(x,\xi)\in \mathcal{S}'(\reals^d \times \reals^d)$ acts as a distribution. Recall that $\mathcal{S}$ is the Schwartz space and $\mathcal{S}'$ is its dual or the space of tempered distributions. The interested reader may refer to \cite{folland95} or \cite{shubin01} for a deeper discourse.

The term ``pseudodifferential'' arises from the fact that differential operators have very simple symbols. For example, the Laplacian has the symbol $a(x,\xi) =-4\pi^2 \norm{\xi}^2$. Another example is
\begin{equation*}
Au(x) = u(x) - \nabla \cdot{\alpha(x) \grad u(x)} \mbox{ for some }\alpha(x)\in C^1(\reals^d).
\end{equation*}
Its symbol is
\begin{equation}\label{eq:example1}
a(x,\xi)=1+\alpha(x)(4\pi^2 \norm{\xi}^2)-\sum_{k=1}^d (2\pi i \xi_k)\pd_{x_k}\alpha(x).
\end{equation}
Clearly, if the media $\alpha(x)$ is smooth, so is the symbol $a(x,\xi)$ smooth in both $x$ and $\xi$, an important property which will be used in Section \ref{sec:symexpand}.

For practical reasons, we make the following assumptions about $u:\reals^d \rightarrow \reals$ on which symbols are applied.
\begin{enumerate}
\item $u$ is periodic with period 1, so only $\xi \in \ints^d$ will be considered in the Fourier domain.
\item $u$ is bandlimited, say $\hat{u}$ is supported on $\Xi:=[-\xi_0,\xi_0]^d \subseteq \ints^d$. Any summation over the Fourier domain is by default over $\Xi$.\footnote{To have an even number of points per dimension, one can use $\Xi=[-\xi_0,\xi_0-1]^d$ for example. We leave this generalization to the reader and continue to assume $\xi\in [-\xi_0,\xi_0]^d$.}
\item $a(x,\xi)$ and $u(x)$ are only evaluated at $x\in X\subset[0,1]^d$ which are points uniformly spaced apart. Any summation over $x$ is by default over $X$.
\end{enumerate}

Subsequently, Equation (\ref{eq:symbolrep}) reduces to a discrete and finite form:
\begin{equation}\label{eq:discretesymbol}
Au(x) = \sum_{\xi \in \Xi} e^{2\pi i \xi \cdot x} a(x,\xi) \hat{u}(\xi).
\end{equation}
We like to call $a(x,\xi)$ a ``discrete symbol.'' Some tools are already available for manipulating such symbols \cite{demanet2010discrete}.

\subsection{User friendly representations of symbols}
Given a linear operator $A$, it is useful to relate its symbol $a(x,\xi)$ to its matrix representation in the Fourier basis. This helps us understand the symbol as a matrix and also exposes easy ways of computing the symbols of $A^{-1}, A^*$ and $AB$ using standard linear algebra software.

By a matrix representation $(A_{\eta \xi})$ in Fourier basis, we mean of course that $\widehat{Au}(\eta) = \sum_{\xi} A_{\eta \xi} \hat{u}(\xi)$ for any $\eta$. We also introduce a more compact form of the symbol: $\hat{a}(j,\xi) = \int_x a(x,\xi) e^{-2\pi i j\cdot x} dx$. The next few results are pedagogical and listed for future reference.

\begin{proposition}\label{thm:a2mat}
Let $A$ be a linear operator with symbol $a(x,\xi)$. Let $(A_{\eta \xi})$ and $\hat{a}(j,\xi)$ be as defined above. Then
$$A_{\eta \xi} = \int_{x} a(x,\xi) e^{-2\pi i (\eta - \xi) x} dx;\quad a(x,\xi) = e^{-2\pi i \xi x} \sum_{\eta} e^{2\pi i \eta x} A_{\eta \xi};$$
$$A_{\eta \xi} = \hat{a}(\eta - \xi,\xi);\quad \hat{a}(j, \xi) = A_{j+\xi,\xi}.$$
\end{proposition}

\begin{proof}
Let $\eta=\xi+j$ and apply the definitions.
\end{proof}

\begin{proposition}[Trace]\label{thm:trace}
Let $A$ be a linear operator with symbol $a(x,\xi)$. Then
$$\tr(A) = \sum_{\xi} \hat{a}(0,\xi) = \sum_{\xi} \int_x a(x,\xi) dx.$$
\end{proposition}

\begin{proposition}[Adjoint]\label{thm:adjoint}
Let $A$ and $C=A^*$ be linear operators with symbols $a(x,\xi),c(x,\xi)$. Then
$$\hat{c}(j,\xi) = \overline{\hat{a}(-j,j+\xi)};\quad c(x,\xi) = \sum_{\eta} \int_y \overline{a(y,\eta)} e^{2\pi i (\eta - \xi)(x-y)}dy.$$
\end{proposition}

\begin{proposition}[Composition]\label{thm:compose}
Let $A,B$ and $C=AB$ be linear operators with symbols $a(x,\xi),b(x,\xi),c(x,\xi)$. Then
$$\hat{c}(j,\xi) = \sum_{\zeta} \hat{a}(j+\xi-\zeta,\zeta)\hat{b}(\zeta-\xi,\xi);$$
$$c(x,\xi)=\sum_{\zeta} \int_y e^{2\pi i (\zeta - \xi)(x-y)} a(x,\zeta) b(y,\xi) dy.$$

\end{proposition}

We leave it to the reader to verify the above results.

\subsection{Symbol expansions}\label{sec:symexpand}

The idea is that when a linear operator $A$ has a smooth symbol $a(x,\xi)$, only a few basis functions are needed to approximate $a$, and correspondingly only a small $\mathcal{B}$ is needed to represent $A$. This is not new, see for example \cite{demanet2010discrete}. In this paper, we consider the separable expansion
\begin{equation*}\label{eq:separablesymbol}
a(x,\xi) = \sum_{j k} c_{jk} e_j(x) g_k(\xi).
\end{equation*}

This is the same as expanding $A$ as $\sum_{jk} c_{jk} B_{jk}$ where the symbol for $B_{jk}$ is $e_j(x) g_k(\xi)$. With an abuse of notation, let $B_{jk}$ also denote its matrix representation \emph{in Fourier basis}. Given our assumption that $\xi\in [-\xi_0,\xi_0]^d$, we have $B_{jk}\in \cplexes^{n \times n}$ where $n=(2\xi_0+1)^d$. As its symbol is separable, $B_{jk}$ can be factorized as
\begin{equation}\label{eq:symfactors}
B_{jk} = \mathcal{F} \diag(e_j(x)) \mathcal{F}^{-1}\diag(g_k(\xi))
\end{equation}
where $\mathcal{F}$ is the unitary Fourier matrix. An alternative way of viewing $B_{jk}$ is that it takes its input $\hat{u}(\xi)$, multiply by $g_k(\xi)$ and convolve it with $\hat{e}_j(\eta)$, the Fourier transform of $e_j(x)$. There is also an obvious algorithm to apply $B_{jk}$ to $u(x)$ in $\tilde{O}(n)$ time as outlined below. As mentioned in Section \ref{sec:motivation}, this speeds up the recovery of the coefficients $c$ and makes matrix probing a cheap operation.

\begin{algorithm}
Given vector $u(x)$, apply the symbol $e_j(x) g_k(\xi)$.
\begin{enumerate}
\item Perform FFT on $u$ to obtain $\hat{u}(\xi)$.
\item Multiply $\hat{u}(\xi)$ by $g_k(\xi)$ elementwise.
\item Perform IFFT on the previous result, obtaining $\sum_{\xi} e^{2\pi i \xi \cdot x} g_k(\xi) \hat{u}(\xi)$.
\item Multiply the previous result by $e_j(x)$ elementwise.
\end{enumerate}
\end{algorithm}

Recall that for $L$ to be well-conditioned with high probability, we need to check whether $N$, as defined in Equation (\ref{eq:Njk}), is well-conditioned, or in a rough sense whether $\inner{B_j}{B_k} \simeq \delta_{jk}$. For separable symbols, this inner product is easy to compute.

\begin{proposition}\label{thm:syminner}
Let $B_{jk},B_{j'k'}\in \cplexes^{n \times n}$ be matrix representations (in Fourier basis) of linear operators with symbols $e_j(x) g_k(\xi)$ and $e_{j'}(x) g_{k'}(\xi)$. Then
$$\inner{B_{jk}}{B_{j'k'}} =\inner{e_j}{e_{j'}}\inner{g_k}{g_{k'}}$$
where $\inner{e_j}{e_{j'}}=\frac{1}{n}\sum_{i=1}^n \overline{e_j(x_i)} e_{j'}(x_i)$ and $x_1,\ldots,x_n$ are points in $[0,1]^d$ uniformly spaced, and $\inner{g_k}{g_{k'}}=\sum_{\xi} \overline{g_k(\xi)} g_k(\xi)$.
\end{proposition}
\begin{proof}
Apply Propositions \ref{thm:trace}, \ref{thm:adjoint} and \ref{thm:compose} with the symbols in the $\hat{a}(\eta,\xi)$ form.
\end{proof}

To compute $\lambda(\mathcal{B})$ as in Definition \ref{def:weakcondition}, we examine the spectrum of $B_{jk}$ for every $j,k$. A simple and relevant result is as follows.
\begin{proposition}\label{thm:symspec1}
Assume the same set-up as in Proposition \ref{thm:syminner}. Then
$$\sigma_{\min}(B_{jk}) \geq \min_x |e_j(x)| \min_{\xi} |g_k(\xi)|;\quad \sigma_{\max}(B_{jk}) \leq \max_x |e_j(x)| \max_{\xi} |g_k(\xi)|.$$
\end{proposition} 
\begin{proof}
In Equation (\ref{eq:symfactors}), $\mathcal{F} \diag(e^j(x)) \mathcal{F}^{-1}$ has singular values $|e_j(x)|$ as $x$ varies over $X$, defined at the end of Section \ref{sec:pdobasics}. The result follows from the min-max theorem.
\end{proof}

As an example, suppose $a(x,\xi)$ is smooth and periodic in both $x$ and $\xi$. It is well-known that a Fourier series is good expansion scheme because the smoother $a(x,\xi)$ is as a periodic function in $x$, the faster its Fourier coefficients decay, and less is lost when we truncate the Fourier series. Hence, we pick\footnote{Actually, $\exp(2\pi i k \xi_0/(2\xi_0+1))$ does not vary with $\xi$, and we can use $\vphi(\xi)=\xi/(2\xi_0+1)$.}
\begin{equation}\label{eq:fourierexpand}
e_j(x) = e^{2\pi i j \cdot x};\quad g_k(\xi)= e^{2\pi i k \cdot \vphi(\xi)},
\end{equation}
where $\vphi(\xi) = (\xi+\xi_0)/(2\xi_0+1)$ maps $\xi$ into $[0,1]^d$.

Due to Proposition \ref{thm:syminner}, $N=\E M$ is a multiple of the identity matrix and $\kappa(\mathcal{B})=1$ where $\mathcal{B}=\{B_{jk}\}$. It is also immediate from Proposition \ref{thm:symspec1} that $\lambda(B_{jk})=1$ for every $j,k$, and $\lambda(\mathcal{B})=1$. The optimal condition numbers of this $\mathcal{B}$ make it suitable for matrix probing.

\subsection{Chebyshev expansion of symbols}\label{sec:chebexpand}
The symbols of differential operators are polynomials in $\xi$ and nonperiodic. When probing these operators, a Chebyshev expansion in $\xi$ is in principle favored over a Fourier expansion, which may suffer from the Gibbs phenomenon. However, as we shall see, $\kappa(\mathcal{B})$ grows with $p$ and can lead to ill-conditioning.

For simplicity, assume that the symbol is periodic in $x$ and that $e_j(x)=e^{2\pi i j\cdot x}$. Applying Proposition \ref{thm:a2mat}, we see that $B_{jk}$ is a matrix with a displaced diagonal and its singular values are $(g_k(\xi))_{\xi\in \Xi}$. (Recall that we denote the matrix representation (in Fourier basis) of $B_{jk}$ as $B_{jk}$ as well.)

Let $T_k$ be the $k$-th Chebyshev polynomial. In 1D, we can pick
\begin{equation}\label{eq:cheb1d}
g_k(\xi) = T_k(\xi/\xi_0) \mbox{ for } k=1,\ldots,K.
\end{equation}

Define $\norm{T_k}_2=(\int_{z=-1}^1 T_k(z)^2 dz)^{1/2}$. By approximating sums with integrals, $\lambda(B_{jk})\simeq \sqrt{2} \norm{T_k}_{2}^{-1}=\left(\frac{4 k^2-1}{2 k^2 -1}\right)^{1/2}$. Notice that there is no $(1-z^2)^{-1/2}$ weight factor in the definition of $\norm{T_k}_2$ because $e_j(x)T_k(\xi)$ is treated as a pseudodifferential symbol and has to be evaluated on the \emph{uniform} grid. In practice, this approximation becomes very accurate with larger $n$ and we see no need to be rigorous here. As $k$ increases, $\lambda(B_{jk})$ approaches $\sqrt{2}$. More importantly, $\lambda(B_{jk})\leq \lambda(B_{j1})$ for any $j,k$, so
$$\lambda(\mathcal{B})=\sqrt{3}.$$

Applying the same technique to approximate the sum $\inner{g_k}{g_{k'}}$, we find that
$\inner{g_k}{g_{k'}} \propto (1-(k+k')^2)^{-1}+(1-(k-k')^2)^{-1}$ when $k+k'$ is even, and zero otherwise. 
We then compute $N=\E M$ for various $K$ and plot $\kappa(\mathcal{B})$ versus $K$, the number of Chebyshev polynomials. As shown in Figure \ref{fig:chebkappa}(a),
$$\kappa(\mathcal{B})\simeq 1.3K.$$

This means that if we expect to recover $p=\tilde{O}(n)$ coefficients, we must keep $K$ fixed. Otherwise, if $p=K^{2}$, only $p=\tilde{O}(n^{1/2})$ are guaranteed to be recovered by Theorem \ref{thm:main}.

\begin{figure}[tb]
\centering
\includegraphics[scale=0.75]{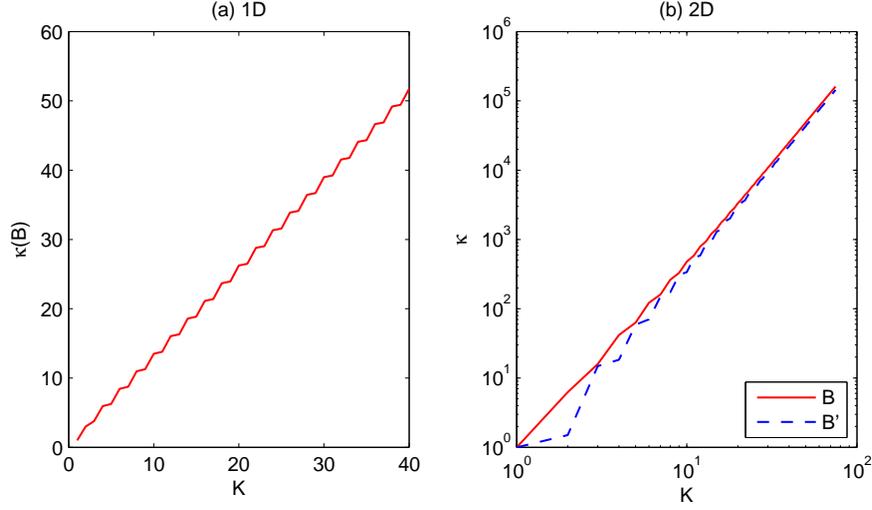} 
\caption{Let $K$ be the number of Chebyshev polynomials used in the expansion of the symbol, see Equation (\ref{eq:cheb1d}) and (\ref{eq:cheb2d}). Observe that in 1D, $\kappa(\mathcal{B})=O(K)$ while in 2D, $\kappa(\mathcal{B}) = O(K^3)$. These condition numbers mean that we cannot expect to retrieve $p=\tilde{O}(n)$ parameters unless $K$ is fixed and independent of $p,n$. \label{fig:chebkappa}}
\end{figure}

In 2D, a plausible expansion is
\begin{equation}\label{eq:cheb2d}
g_{k}(\xi)=e^{i k_1 \arg \xi} T_{k_2}(\vphi(\norm{\xi})) \mbox{ for } 1\leq k_2 \leq K
\end{equation}
where $k=(k_1,k_2)$ and $\vphi(r)=(\sqrt{2}r/\xi_0)-1$ maps $\norm{\xi}$ into $[-1,1]$. We call this the ``Chebyshev on a disk'' expansion.

The quantity $\lambda(B_{jk})$ is approximately $2\left(\int_{x=-1}^1 \int_{y=-1}^1 T_k(\psi(x,y))^2 dx \, dy\right)^{-1/2}$ 
where $\psi(x,y)=(2x^2+2y^2)^{1/2}-1$. The integral is evaluated numerically and appears to converge\footnote{This is because when we truncate the disk of radius $\xi_0\sqrt{2}$ to a square of length $2\xi_0$, most is lost along the vertical axis and away from the diagonals. However, for large $k$, $T_k$ oscillates very much and the truncation does not matter. If we pretend that the square is a disk, then we are back in the 1D case where the answer approaches $\sqrt{2}$ for large $k$.}  to $\sqrt{2}$ for large $k_2$. Also, $k_2=1$ again produces the worst $\lambda(B_{jk})$ and
$$\lambda(\mathcal{B})\leq 2.43.\footnote{The exact value is $2(4-\frac{8}{3}\sqrt{2} \sinh^{-1}(1))^{-1/2}.$}$$

As for $\kappa(\mathcal{B})$, observe that when $k_1 \neq k'_1$, $\inner{g_{k_1 k_2}}{g_{k'_1 k'_2}}=\pm 1$ due to symmetry\footnote{The $\xi$ and $-\xi$ terms cancel each other. Only $\xi=0$ contributes to the sum.}, whereas when $k_1=k'_1$, the inner product is proportional to $n$ and is much larger. As a result, the $g_k$'s with different $k_1$'s hardly interact and in studying $\kappa(\mathcal{B})$, one may assume that $k_1=k'_1=0$. To improve $\kappa(\mathcal{B})$, we can normalize $g_k$ such that the diagonal entries of $N$ are all ones, that is $g'_k(\xi) = g_k(\xi)/\norm{g_k(\xi)}$.

This yields another set of basis matrices $\mathcal{B}'$. Figure \ref{fig:chebkappa}(b) reveals that $$\kappa(\mathcal{B})=O(K^{3}) \mbox{ and }\kappa(\mathcal{B}')\simeq \kappa(\mathcal{B}).$$

The latter can be explained as follows: we saw earlier that $\inner{B_{jk}}{B_{jk}}$ converges as $k_2$ increases, so the diagonal entries of $N$ are about the same and the normalization is only a minor correction.

If $a(x,\xi)$ is expanded using the same number of basis functions in each direction of $x$ and $\xi$, i.e., $K=p^{1/4}$, then Theorem \ref{thm:main} suggests that only $p=\tilde{O}(n^{2/5})$ coefficients can be recovered.

To recap, for both 1D and 2D, $\lambda(\mathcal{B})$ is a small number but $\kappa(\mathcal{B})$ increases with $K$. Fortunately, if we know that the operator being probed is a second order derivative for example, we can fix $K=2$. 

Numerically, we have observed that the Chebyshev expansion can produce dramatically better results than the Fourier expansion of the symbol. More details can be found in Section \ref{sec:2d}.

\subsection{Order of an operator}\label{sec:order}

In standard texts, $A$ is said to be a pseudodifferential operator of order $w$ if its symbol $a(x,\xi)$ is in $C^{\infty}(\reals^d \times \reals^d)$ and for any multi-indices $\alpha,\beta$, there exists a constant $C_{\alpha\beta}$ such that
$$|\pd_{\xi}^{\alpha} \pd_x^{\beta} a(x,\xi) |\leq C_{\alpha \beta} \brac{\xi}^{w-|\alpha|} \mbox{ for all }\xi, \mbox{ where } \brac{\xi}=1+\norm{\xi}.$$

Letting $\alpha=\beta=0$, we see that such operators have symbols that grow or decay as $(1+\norm{\xi})^w$. As an example, the Laplacian is of order 2. The factor 1 prevents $\brac{\xi}$ from blowing up when $\xi=0$. There is nothing special about it and if we take extra care when evaluating the symbol at $\xi=0$, we can use 
$$\brac{\xi} = \norm{\xi}.$$

For forward matrix probing, if it is known a priori that $a(x,\xi)$ behaves like $\brac{\xi}^w$, it makes sense to expand $a(x,\xi) \brac{\xi}^{-w}$ instead. Another way of viewing this is that the symbol of the operator $B_{jk}$ is modified from $e_j(x) g_k(\xi)$ to $e_j(x) g_k(\xi) \brac{\xi}^{w}$ to suit $A$ better.

For backward matrix probing, if $A$ is of order $z$, then $A^{-1}$ is of order $-z$ and we should replace the symbol of $B_{jk}$ with $e_j(x) g_k(\xi) \brac{\xi}^{-w}$. We believe that this small correction has an impact on the accuracy of matrix probing, as well as the condition numbers $\kappa(\mathcal{B}_A)$ and $\lambda(\mathcal{B}_A)$.

Recall that an element of $\mathcal{B}_A$ is $B_{jk} A$. If $A$ is of order $w$ and $B_{jk}$ is of order 0, then $B_{jk} A$ is of order $w$ and $\lambda(B_{jk} A)$ will grow with $n^w$, which will adversely affect the conditioning of matrix probing. However, by multiplying the symbol of $B_{jk}$ by $\brac{\xi}^{-w}$, we can expect $B_{jk} A$ to be order 0 and that $\lambda(B_{jk} A)$ is independent of the size of the problem $n$. The argument is heuristical but we will support it with some numerical evidence in Section \ref{sec:2d}.

\section{Numerical examples}\label{sec:numerical}

We carry out four different experiments. The first experiment suggests that Theorem \ref{thm:mainweak} is not tight. The second experiment presents the output of backward probing in a visual way. In the third experiment, we explore the limitations of backward probing and also tests the Chebyshev expansion of symbols. The last experiment involves the forward probing of the foveation operator, which is related to human vision.

\subsection{1D statistical study}\label{sec:1Dstatstudy}

We are interested in whether the probability bound in Theorem \ref{thm:main} is tight with respect to $p$ and $n$, but as the tail probabilities are small and hard to estimate, we opt to study the first moment instead. In particular, if Theorem \ref{thm:mainweak} captures exactly the dependence of $\E \norm{M-N}/\norm{N}$ on $p$ and $n$, then we would need $n$ to grow faster than $p \log^2 n$ for $\E \norm{M-N}/\norm{N}$ to vanish, assuming $\lambda(\mathcal{B})$ is fixed.

For simplicity, we use the Fourier expansion of the symbol in 1D so that $\lambda(\mathcal{B})=\kappa(\mathcal{B})=1$. Let $J$ be the number of basis functions in both $x$ and $\xi$ and $p=J^2$. Figure \ref{fig:logpowercombo}(a) suggests that $\E \norm{M-N}/\norm{N}$ decays to zero when $n=p \log^c p$ and $c>1$. It follows from the previous paragraph that Theorem \ref{thm:mainweak} cannot be tight.

\begin{figure}[tb]
\centering
\includegraphics[scale=0.75]{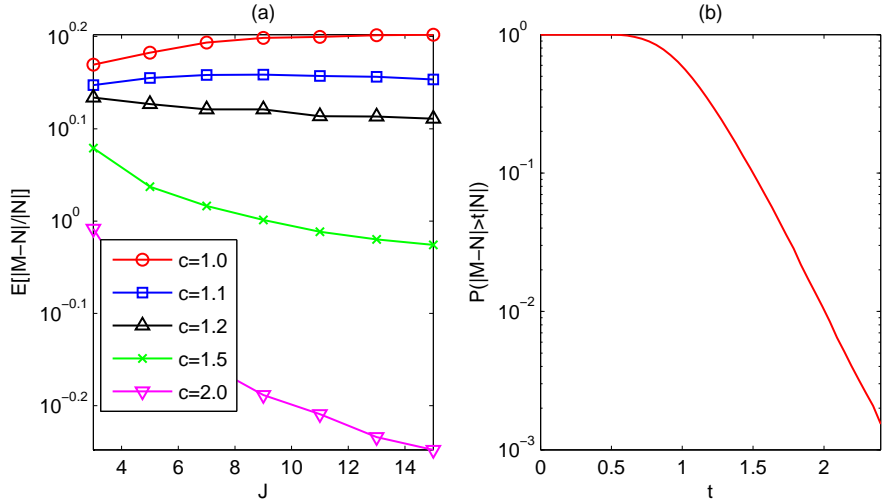} 
\caption{\label{fig:logpowercombo}Consider the Fourier expansion of the symbol. $J$ is the number of basis functions in $x$ and $\xi$, so $p=J^2$. Let $n=p\log^c p$. Figure (a) shows that the estimated $\E \norm{M-N}/\norm{N}$ decays for $c\geq 1.1$ which suggests that Theorem \ref{thm:mainweak} is not tight. In Figure (b), we estimate $\P{\norm{M-N}/\norm{N}>t}$ by sampling $\norm{M-N}/\norm{N}$ $10^5$ times. The tail probability appears to be subgaussian for small $t$ and subexponential for larger $t$.}
\end{figure}

Nevertheless, Theorem \ref{thm:mainweak} is optimal in the following sense. Imagine a more general bound
\begin{equation}\label{eq:boundform}
\E \frac{\norm{M-N}}{\norm{N}} \leq (\log^{\alpha} n) \left(\frac{p}{n} \right)^{\beta} \mbox{ for some } \alpha,\beta>0.
\end{equation}

\begin{figure}[tb]
\centering
\includegraphics[scale=0.75]{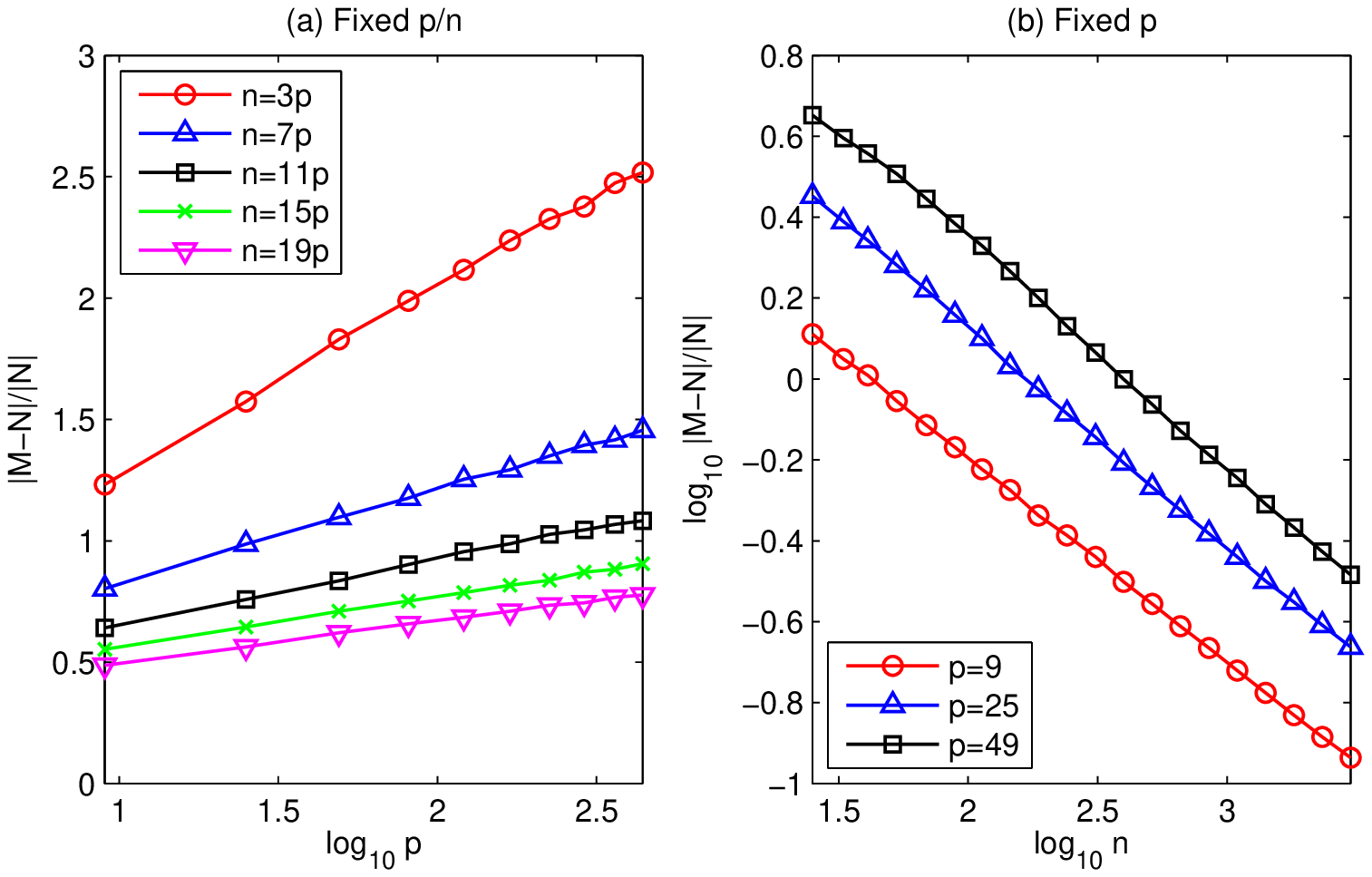} 
\caption{\label{fig:stats}Consider bounding $\E \norm{M-N}/\norm{N}$ by $(\log^{\alpha} n) (p/n)^{\beta}$. There is little loss in replacing $\log n$ with $\log p$ in the simulation. In Figure (a), the estimated $\E \norm{M-N}/\norm{N}$ depends linearly on $\log p$, so $\alpha \geq 1$. In Figure (b), we fix $p$ and find that for large $n$, $\beta=1/2$. The conclusion is that the bound in Theorem \ref{thm:mainweak} has the best $\alpha,\beta$.}
\end{figure}

In Figure \ref{fig:stats}(a), we see that for various values of $p/n$, $\alpha=1$ since the graphs are linear. On the other hand, if we fix $p$ and vary $n$, the log-log graph of Figure \ref{fig:stats}(b) shows that $\beta=1/2$. Therefore, any bound in the form of Equation (\ref{eq:boundform}) is no better than Theorem \ref{thm:mainweak}.

Next, we fix $p=25,n=51$ and sample $\norm{M-N}/\norm{N}$ many times to estimate the tail probabilities. In Figure \ref{fig:logpowercombo}(b), we see that the tail probability of $\P{\norm{M-N}/\norm{N}>t}$ decays as $\exp(-c_1t)$ when $t$ is big, and as $\exp(-c_2 t^2)$ when $t$ is small, for some positive numbers $c_1,c_2$. This behavior may be explained by Rauhut and Tropp's yet published result.

\subsection{Elliptic equation in 1D}\label{sec:1d}

We find it instructive to consider a 1D example of matrix probing because it is easy to visualize the symbol $a(x,\xi)$. Consider the operator
\begin{equation}\label{eq:op1d}
Au(x) = -\frac{d}{dx} \alpha(x) \frac{du(x)}{dx} \mbox{ where } \alpha(x) = 1+0.4\cos(4\pi x) + 0.2 \cos(6\pi x).
\end{equation}

Note that we use periodic boundaries and $A$ is positive semidefinite with a one dimensional nullspace consisting of constant functions.

We probe for $A^{+}$ using Algorithm \ref{alg:probeinv} and the Fourier expansion of its symbol or Equation (\ref{eq:fourierexpand}). Since $A$ is of order $2$, we premultiply $g_k(\xi)$ by $\brac{\xi}^{-2}$ as explained in Section \ref{sec:order}.

\begin{figure}[tb]
\centering
\includegraphics[scale=0.75]{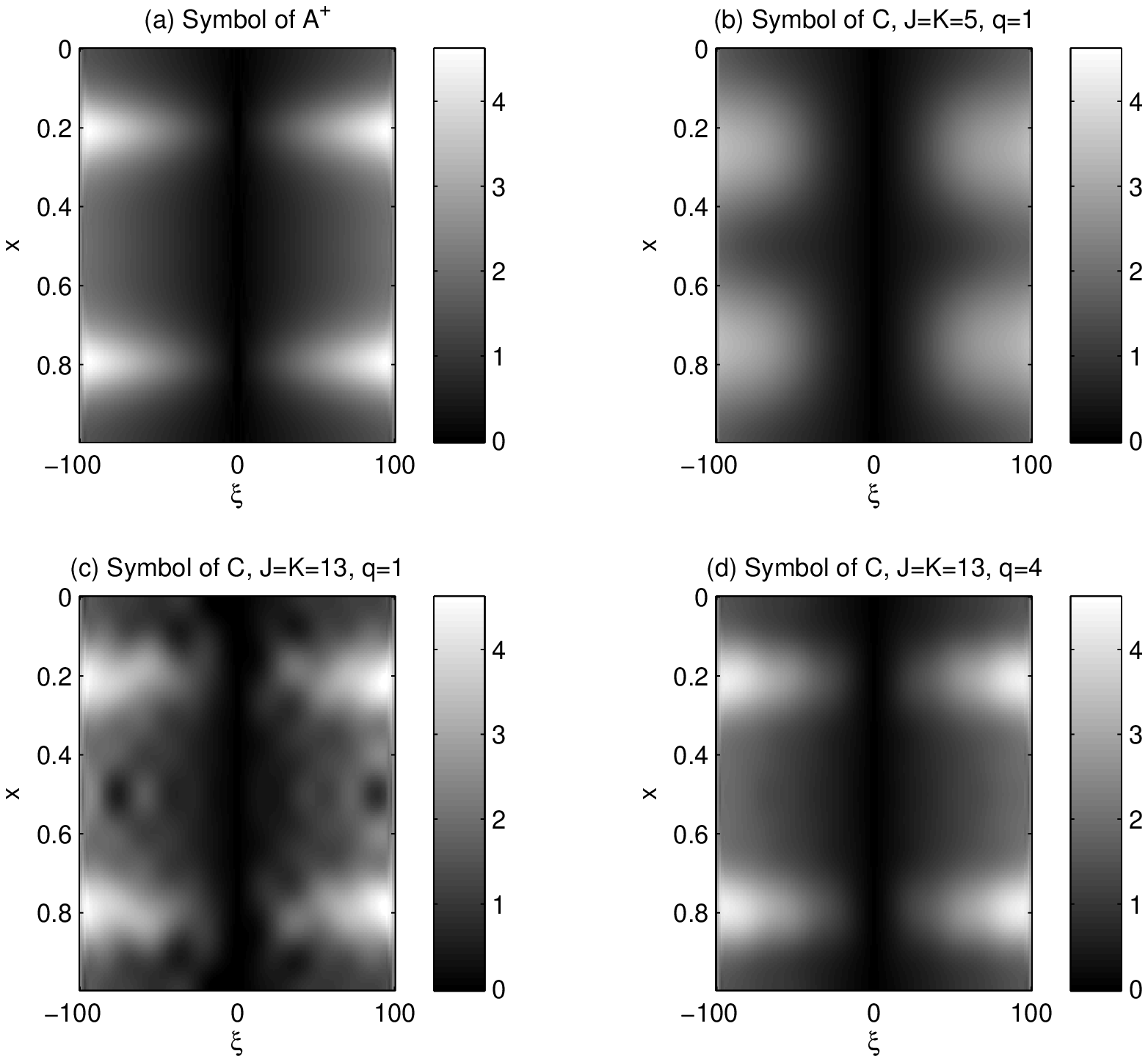} 
\caption{Let $A$ be the 1D elliptic operator in Equation (\ref{eq:op1d}) and $A^{+}$ be its pseudoinverse. Let $C$ be the output of backward matrix probing with the following parameters: $q$ is the number of random vectors applied to $A^{+}$; $J,K$ are the number of $e_j$'s and $g_k$'s used to expand the symbol of $A^{+}$ in Equation (\ref{eq:fourierexpand}). Figure (a) is the symbol of $A^+$. Figure (b) is the symbol of $C$ with $J=K=5$. It lacks the sharp features of Figure (a) because $\mathcal{B}$ is too small to represent $A^{+}$ well. With $J=K=13$, probing with only one random vector leads to ill-conditioning and an inaccurate result in Figure (b). In Figure (c), four random vectors are used and a much better result is obtained. Note that the symbols are multipled by $\brac{\xi}^3$ for better visual contrast. \label{fig:symbol1D}}
\end{figure}

In the experiment, $n=201$ and there are two other parameters $J,K$ which are the number of $e_j$'s and $g_k$'s used in Equation (\ref{eq:fourierexpand}). To be clear, $-\frac{J-1}{2} \leq j \leq \frac{J-1}{2}$ and $-\frac{K-1}{2} \leq k \leq \frac{K-1}{2}$.

Let $C$ be the output of matrix probing. In Figure \ref{fig:symbol1D}(b), we see that $J=K=5$ is not enough to represent $A^{+}$ properly. This is expected because our media $\alpha(x)$ has a bandwidth of 7. We expect $J=K=13$ to do better, but the much larger $p$ leads to overfitting and a poor result, as is evident from the wobbles in the symbol of $C$ in Figure \ref{fig:symbol1D}(c). Probing with four random vectors, we obtain a much better result as shown in Figure \ref{fig:symbol1D}(d).

\subsection{Elliptic Equation in 2D}\label{sec:2d}
In this section, we extend the previous set-up to 2D and address a different set of questions. Consider the operator $A$ defined as
\begin{equation}\label{eq:2delliptic}
Au(x) = -\nabla \cdot \alpha(x) \nabla u(x) \mbox{ where } \alpha(x) = \frac{1}{T}+\cos^2 (\pi \gamma x_1) \sin^2 (\pi \gamma x_2).
\end{equation}

The positive value $T$ is called the contrast while the positive integer $\gamma$ is the roughness of the media, since the bandwidth of $\alpha(x)$ is $2\gamma+1$. Again, we assume periodic boundary conditions such that $A$'s nullspace is exactly the set of constant functions.

Let $C$ be the output of the backward probing of $A$. As we shall see, the quality of $C$ drops as we increase the contrast $T$ or the roughness $\gamma$.

Fix $n=101^2$ and expand the symbol using Equation (\ref{eq:fourierexpand}). Let $J=K$ be the number of basis functions used to expand the symbol in each of its four dimensions, that is $p=J^4$.

\begin{figure}[tb]
\centering
\includegraphics[scale=0.75]{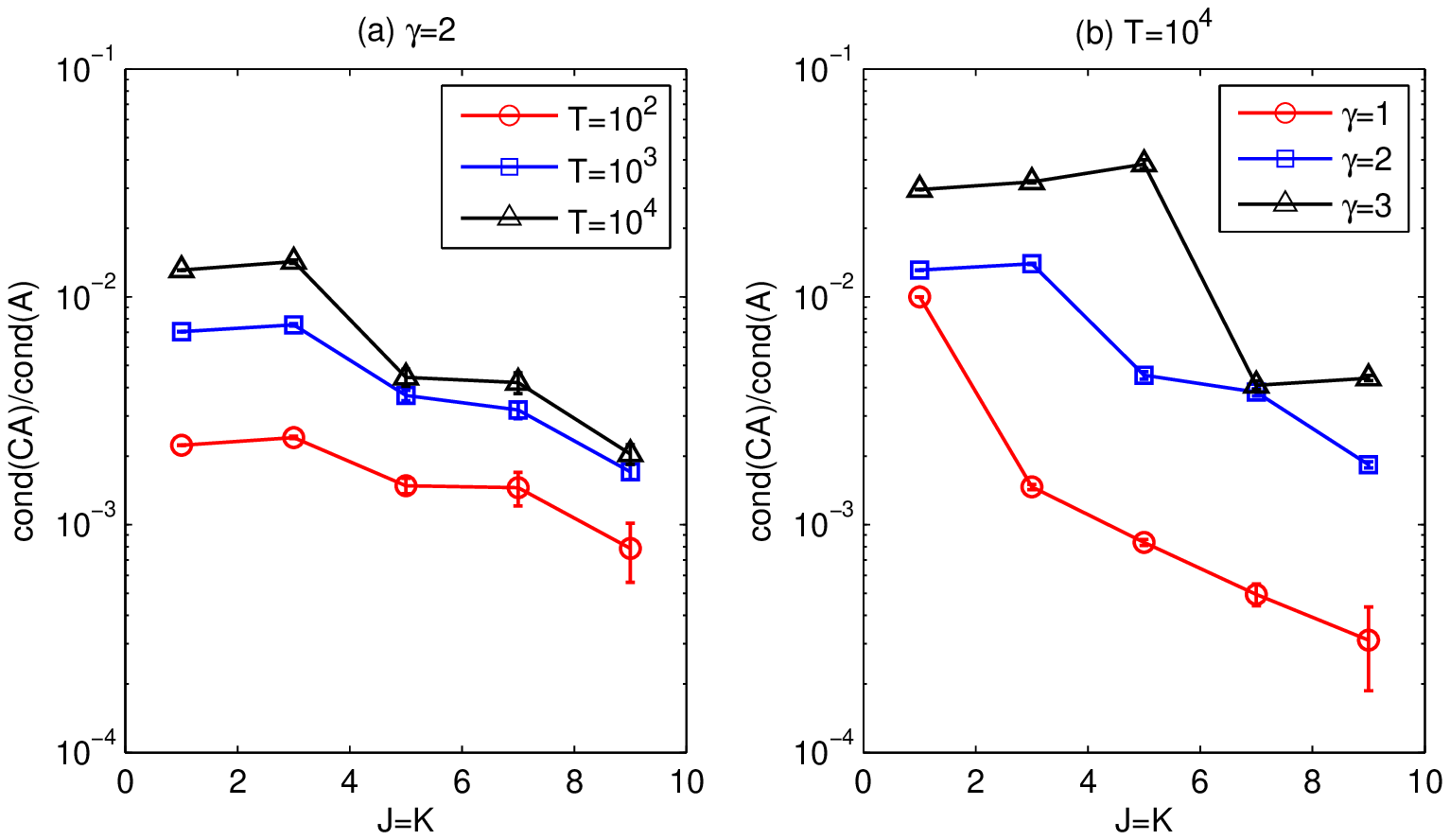} 
\caption{\label{fig:roughness_contrast}Let $A$ be the operator defined in Equation (\ref{eq:2delliptic}) and $C$ be the output of backward probing. In Figure (b), we fix $T=10^4$ and find that as $J$ goes from $2\gamma-1$ to $2\gamma+1$, the bandwidth of the media, the quality of the preconditioner $C$ improves by a factor between $10^{0.5}$ and $10$. In Figure (a), we fix $\gamma=2$ and find that increasing the contrast worsens $\cond(CA)/\cond(A)$. Nevertheless, the improvement between $J=3$ and $J=5$ becomes more distinct. The error bars correspond to $\hat{\sigma}$ where $\hat{\sigma}^2$ is the estimated variance. They indicate that $C$ is not just good on average, but good with high probability.}
\end{figure}

In Figure \ref{fig:roughness_contrast}(b), we see that between $J=2\gamma-1$ and $J=2\gamma+1$, the bandwidth of the media, there is a marked improvement in the preconditioner, as measured by the ratio $\cond(CA)/\cond(A)$.\footnote{Since $A$ has one zero singular value, $\cond(A)$ actually refers to the ratio between its largest singular value and its second smallest singular value. The same applies to $CA$.}

On the other hand, Figure \ref{fig:roughness_contrast}(a) shows that as the contrast increases, the preconditioner $C$ degrades in performance, but the improvement between $J=2\gamma-1$ and $2\gamma+1$ becomes more pronounced.


The error bars in Figure \ref{fig:roughness_contrast} are not error margins but $\hat{\sigma}$ where $\hat{\sigma}^2$ is the unbiased estimator of the variance. They indicate that $\cond(CA)/\cond(A)$ is tightly concentrated around its mean, provided $J$ is not too much larger than is necessary. For instance, for $\gamma=1$, $J=3$ already works well but pushing to $J=9$ leads to greater uncertainty.

Next, we consider \emph{forward probing} of $A$ using the ``Chebyshev on a disk'' expansion or Equation (\ref{eq:cheb2d}). Let $m$ be the order correction, that is we multiply $g_k(\xi)$ by $\brac{\xi}^m =\norm{\xi}^m$. Let $C$ be the output of the probing and $K$ be the number of Chebyshev polynomials used.

\begin{figure}[tb]
\centering
\includegraphics[scale=0.75]{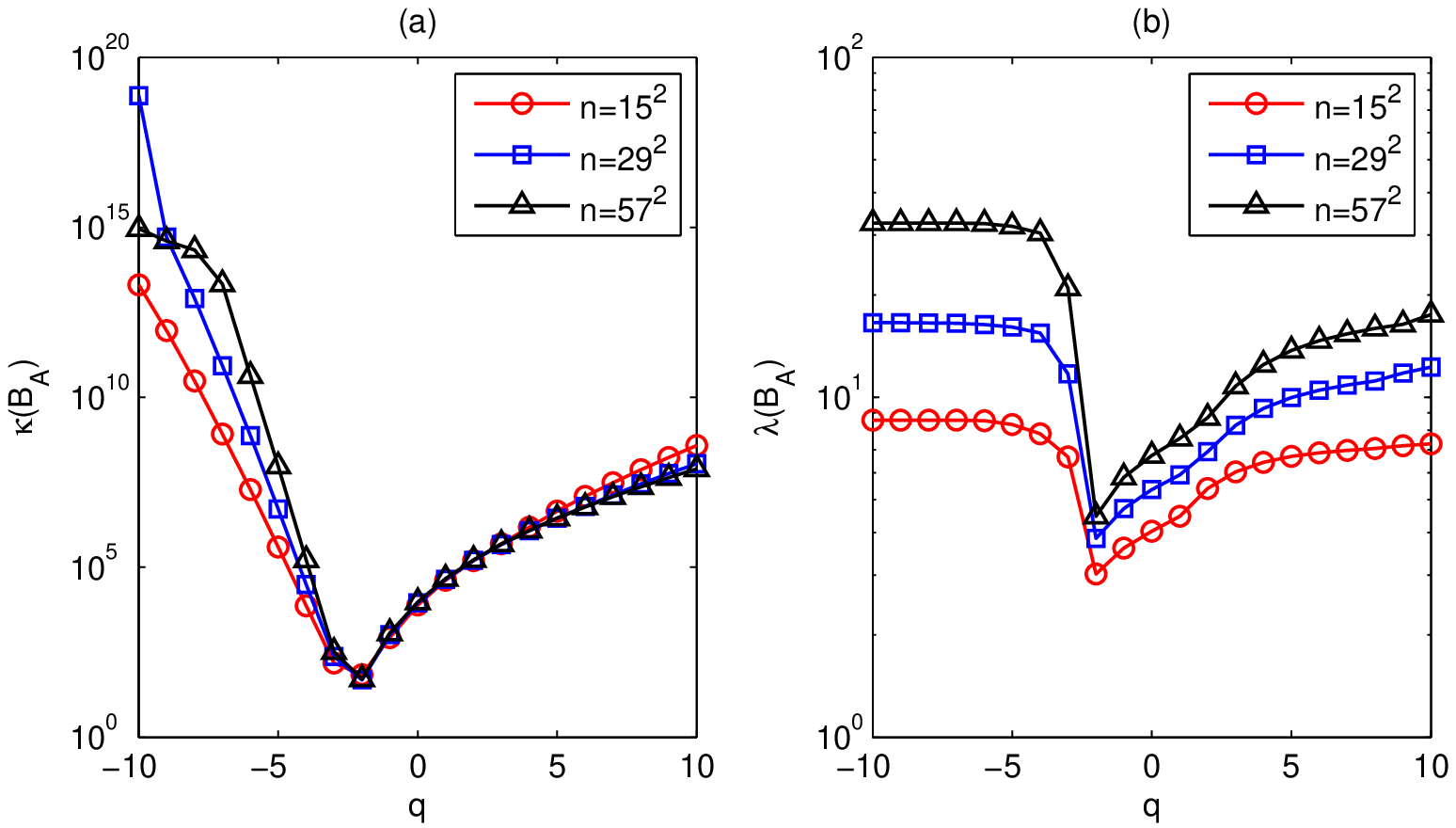} 
\caption{\label{fig:kappalambda} Consider the backward probing of $A$ in Equation (\ref{eq:2delliptic}), a pseudodifferential oeprator of order 2. Perform order correction by multiplying $g_k(\xi)$ by $\brac{\xi}^q$ in the expansion of the symbol. See Section \ref{sec:order}. Observe that at $q=-2$, the condition numbers $\lambda(\mathcal{B}_A)$ and $\kappa(\mathcal{B}_A)$ are minimized and hardly grow with $n$.}
\end{figure}

Fix $n=55^2$, $T=10$, $\gamma=2$ and $J=5$. For $m=0$ and $K=3$, i.e., no order correction and using up to quadratic polynomials in $\xi$, we obtain a relative error $\norm{C-A}/\norm{A}$ that is less than $10^{-14}$. On the other hand, using Fourier expansion, with $K=5$ in the sense that $-\frac{K-1}{2}\leq k_1,k_2 \leq \frac{K-1}{2}$, the relative error is on the order of $10^{-1}$. The point is that in this case, $A$ has an exact ``Chebyshev on a disk'' representation and probing using the correct $\mathcal{B}$ enables us to retrieve the coefficients with negligible errors.

Finally, we consider backward probing with the Chebyshev expansion. We use $J=5$, $\gamma=2$ and $T=10$. Figure \ref{fig:kappalambda} shows that when $m=-2$, the condition numbers $\lambda(\mathcal{B}_A)$ and $\kappa(\mathcal{B}_A)$ are minimized and hardly increases with $n$. This emphasizes the importance of knowing the order of the operator being probed.

\subsection{Foveation}
In this section, we forward-probe for the foveation operator, a space-variant imaging operator \cite{chang00}, which is particularly interesting as a model for human vision. Formally, we may treat the foveation operator $A$ as a Gaussian blur with a width or standard deviation that varies over space, that is
\begin{equation}\label{eq:fov}
Au(x) = \int_{\reals^2}K(x,y) u(y)dy \mbox{ where } K(x,y)=\frac{1}{w(x)\sqrt{2\pi}}\exp\left( \frac{-\norm{x-y}^2}{2 w^2(x)}\right),
\end{equation}
where $w(x)$ is the width function which returns only positive real numbers.

The resolution of the output image is highest at the point where $w(x)$ is minimal. Call this point $x_0$. It is the point of fixation, corresponding to the center of the fovea. For our experiment, the width function takes the form of $w(x) = (\alpha\norm{x-x_0}^2 +\beta)^{1/2}$. Our images are $201\times 201$ and treated as functions on the unit square. We choose $x_0=(0.5,0.5)$ and $\alpha,\beta>0$ such that $w(x_0)=0.003$ and $w(1,1)=0.012$.

The symbol of $A$ is $a(x,\xi)=\exp(-2\pi^2 w(x)^2\norm{\xi}^2)$, and we choose to use a Fourier series or Equation (\ref{eq:fourierexpand}) for expanding it. Let $C$ be the output of matrix probing and $z$ be a standard test image. Figure \ref{fig:fov}(c) shows that the relative $\ell^2$ error $\norm{Cz-Az}_{\ell^2}/\norm{Az}_{\ell^2}$ decreases exponentially as $p$ increases. In general, forward probing yields great results like this because we know its symbol well and can choose an appropriate $\mathcal{B}$.

\begin{figure}[tb]
\centering
\includegraphics{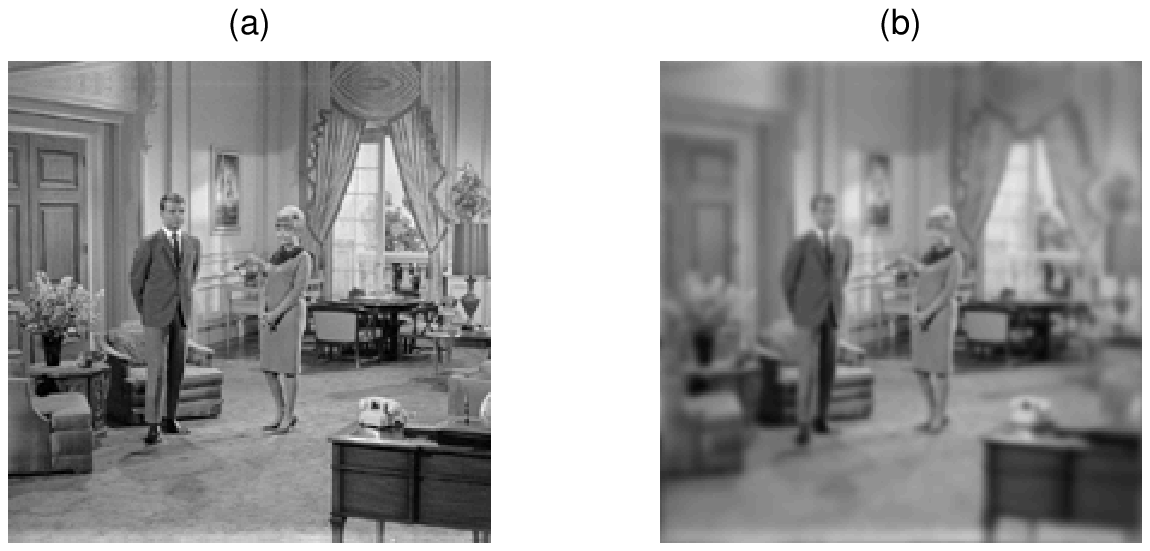} 
\includegraphics[scale=0.75]{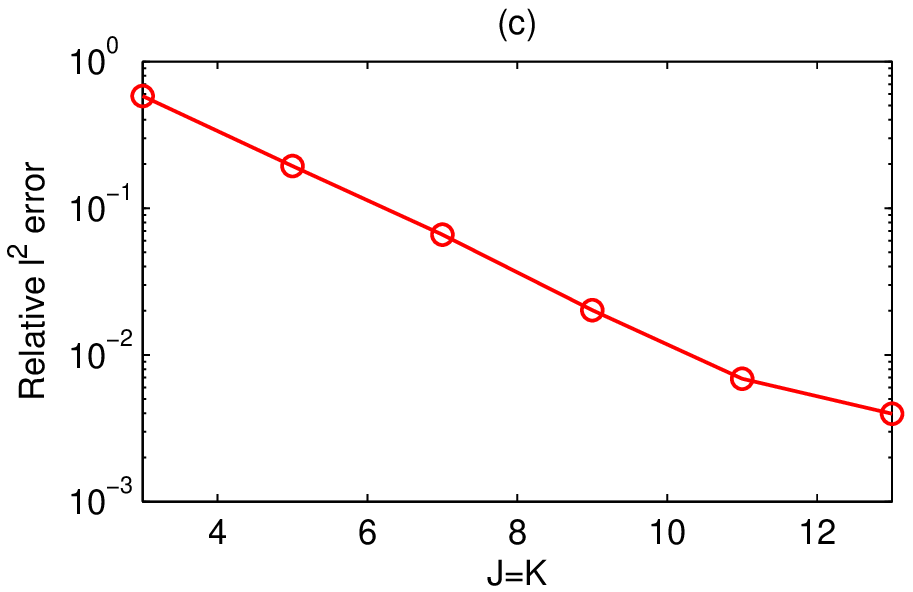} 
\caption{Let $A$ be the foveation operator in Equation (\ref{eq:fov}) and $C$ be the output of the forward probing of $A$. Figure (a) is the test image $z$. Figure (b) is $Cz$ and it shows that $C$ behaves like the foveation operator as expected. Figure (c) shows that the relative $\ell^2$ error (see text) decreases exponentially with the number of parameters $p=J^4$. \label{fig:fov}}
\end{figure}

\subsection{Inverting the wave equation Hessian}\label{sec:hessian}
In seismology, it is common to recover the model parameters $m$, which describe the subsurface, by minimizing the least squares misfit between the observed data and $F(m)$ where $F$, the forward model, predicts data from $m$.

Methods to solve this problem can be broadly categorized into two classes: steepest descent or Newton's method. The former takes more iterations to converge but each iteration is computationally cheaper. The latter requires the inversion of the Hessian of the objective function, but achieves quadratic convergence near the optimal point.

In another paper, we use matrix probing to precondition the inversion of the Hessian. Removing the nullspace component from the noise vector is more tricky (see Algorithm \ref{alg:probeinv}) and involves checking whether ``a curvelet is visible to any receiver'' via raytracing. For details on this more elaborate application, please refer to \cite{demanet2011matrix}.

\section{Conclusion and future work}
When a matrix $A$ with $n$ columns belongs to a specified $p$-dimensional subspace, say $A=\sum_{i=1}^p c_i B_i$, we can probe it with a few random vectors to recover the coefficient vector $c$.

Let $q$ be the number of random vectors used, $\kappa$ be the condition number of the Gram matrix of $B_1,\ldots,B_p$ and $\lambda$ be the ``weak condition number'' of each $B_i$ (cf. Definition \ref{def:weakcondition}) which is related to the numerical rank. From Theorem \ref{thm:main} and Section \ref{sec:multipleprobes}, we learn that when $nq \propto p(\kappa \lambda \log n)^2$, then the linear system that has to be solved to recover $c$ (cf. Equation (\ref{eq:tosolve})) will be well-conditioned with high probability.  Consequently, the reconstruction error is small by Proposition \ref{thm:accuracy}.

The same technique can be used to compute an approximate $A^{-1}$, or a preconditioner for inverting $A$. In \cite{demanet2011matrix}, we used it to invert the wave equation Hessian --- here we demonstrate that it can also be used to invert elliptic operators in smooth media (cf. Sections \ref{sec:1d} and \ref{sec:2d}).

Some possible future work include the following.
\begin{enumerate}
\item Extend the work of Pfander, Rauhut et. al. \cite{pfander2008identification, pfander2010sparsity,pfander2011restricted}. These papers are concerned with sparse signal recovery. They consider the special case where $\mathcal{B}$ contains $n^2$ matrices each representing a time-frequency shift, but $A$ is an unknown linear combination of only $p$ of them. The task is to identify these $p$ matrices and the associated coefficients by applying $A$ to noise vectors. Our proofs may be used to establish similar recovery results for a more general $\mathcal{B}$. However, note that in \cite{pfander2010sparsity}, Pfander and Rauhut show that $n \propto p \log n$ suffices, whereas our main result requires an additional log factor.

\item Build a framework for probing $f(A)$ interpreted as a Cauchy integral
$$f(A)=\frac{1}{2\pi i } \oint_{\Gamma} f(z) (zI-A)^{-1} dz,$$
where $\Gamma$ is a closed curve enclosing the eigenvalues of $A$. For more on approximating matrix functions, see \cite{hale:2505, higham2008functions}.

\item Consider expansion schemes for symbols that highly oscillate or have singularities that are well-understood.
\end{enumerate}

\appendix
\begin{appendices}
\section{Linear algebra}

Recall the definitions of $\kappa(\mathcal{B})$ and $\lambda(\mathcal{B})$ at the beginning of the paper. The following concerns probing with multiple vectors (cf. Section \ref{sec:multipleprobes}).

\begin{proposition}\label{thm:condunchanged}
Let $I_q\in \cplexes^{q\times q}$ be the identity. Let $\mathcal{B}=\{B_1,\ldots,B_p\}$. Let $B'_j = I_q \otimes B_j$ and $\mathcal{B}'=\{B'_1,\ldots,B'_p\}$. Then $\kappa(\mathcal{B})=\kappa(\mathcal{B}')$ and $\lambda(\mathcal{B})=\lambda(\mathcal{B}')$.
\end{proposition}
\begin{proof}
Define $N\in \cplexes^{p \times p}$ such that $N_{jk}=\inner{B_j}{B_k}$. Define $N'\in \cplexes^{p \times p}$ such that $N'_{jk}=\inner{B'_j}{B'_k}$. Clearly, $N'=q N$, so their condition numbers are the same and $\kappa(\mathcal{B})=\kappa(\mathcal{B}')$. 

For any $A=B_j \in \cplexes^{m \times n}$ and $A'=B_j'$, we have $\frac{\norm{A'}(n q)^{1/2}}{\norm{A'}_F}=\frac{\norm{A}(n q )^{1/2}}{\norm{A}_F q^{1/2}}=\frac{\norm{A}n^{1/2}}{\norm{A}_F}$. Hence, $\lambda(\mathcal{B})=\lambda(\mathcal{B}')$. 
\end{proof}

\section{Probabilistic tools}\label{sec:prob}
In this section, we present some probabilistic results used in our proofs. The first theorem is used to decouple homogeneous Rademacher chaos of order 2 and can be found in \cite{pena99, rauhut09} for example.

\begin{theorem}\label{thm:decoupleR}
Let $(u_i)$ and $(\tilde{u}_i)$ be two iid sequences of real-valued random variables and $A_{ij}$ be in a Banach space where  $1\leq i,j\leq n$. There exists universal constants $C_1,C_2>0$ such that for any $s\geq 1$,
\begin{equation}\label{eq:decouple}
\left(\E \norm{\sum_{1\leq i\neq j\leq n} u_i u_j A_{ij}}^s\right)^{1/s} \leq C_1 C_2^{1/s} \left(\E \norm{\sum_{1\leq i, j \leq n} u_i \tilde{u}_j A_{ij}}^s\right)^{1/s}.
\end{equation}
\end{theorem}

A homogeneous \emph{Gaussian} chaos is one that involves only products of Hermite polynomials with the same total degree. For instance, a homogeneous Gaussian chaos of order 2 takes the form $\sum_{1\leq i \neq j \leq n} g_i g_j A_{ij} + \sum_{i=1}^n (g_i^2-1) A_{ii}$. It can be decoupled according to Arcones and Gin{\'e} \cite{arcones93}.

\begin{theorem}\label{thm:decoupleG}
Let $(u_i)$ and $(\tilde{u}_i)$ be two iid Gaussian sequences and $A_{ij}$ be in a Banach space where  $1\leq i,j\leq n$. There exists universal constants $C_1,C_2>0$ such that for any $s\geq 1$,
\begin{equation*}
\left(\E \norm{\sum_{1\leq i\neq j\leq n} u_i u_j A_{ij}+\sum_{i=1}^n (u_i^2-1) A_{ii}}^s\right)^{1/s} \leq C_1 C_2^{1/s} \left(\E \norm{\sum_{1\leq i,j \leq s} u_i \tilde{u}_j A_{ij}}^s\right)^{1/s}.
\end{equation*}
\end{theorem}
\begin{remark}\label{remark:C1C2}
For Rademacher chaos, $C_1=4$ and $C_2=1$. For Gaussian chaos, we can integrate Equation (2.6) of \cite{arcones93} (with $m=2$) to obtain $C_1=2^{1/2}$ and $C_2=2^{14}$. Better constants may be available.
\end{remark}

We now proceed to the Khintchine inequalties. Let $\norm{\cdot}_{C_s}$ denote the $s$-Schatten norm. Recall that $\norm{A}_{C_s} = (\sum_i |\sigma_i|^s)^{1/s}$ where $\sigma_i$ is a singular value of $A$. The following is due to Lust-Piquard and Pisier \cite{lust86,lust91}.
\begin{theorem}\label{thm:khintchineone}
Let $s\geq 2$ and $(u_i)$ be a Rademacher or Gaussian sequence. Then for any set of matrices $\{A_i\}_{1\leq i \leq n}$,
$$\left(\E \norm{\sum_{i=1}^n u_i A_{i} }_{C_s}^s \right)^{1/s}\leq s^{1/2} \max \left( \norm{(\sum_{i=1}^n A_{i}^* A_{i})^{1/2}}_{C_s}, \norm{(\sum_{i=1}^n A_{i} A_{i}^*)^{1/2}}_{C_s}\right).$$
\end{theorem}
The factor $s^{1/2}$ above is not optimal. See for example \cite{buchholz2001operator} by Buchholz, or \cite{rauhut2010compressive, tropp2008conditioning}.

In \cite{rauhut09}, Theorem \ref{thm:khintchineone} is applied twice in a clever way to obtain a Khintchine inequality for a decoupled chaos of order 2.

\begin{theorem}\label{thm:khintchinetwo}
Let $s \geq 2$ and $(u_i)$ and $(\tilde{u}_i)$ be two independent Rademacher or Gaussian sequences. For any set of matrices $\{A_{ij}\}_{1\leq i,j\leq n}$,
$$\left(\E \norm{\sum_{{1\leq i,j\leq n}} u_i \tilde{u}_j A_{ij} }_{C_s}^s \right)^{1/s}\leq 2^{1/s} s \max ( \norm{Q^{1/2}}_{C_s}, \norm{R^{1/2}}_{C_s}, \norm{F}_{C_{s}},\norm{G}_{C_{s}})$$
where $Q=\sum_{1\leq i,j\leq n} A_{ij}^* A_{ij}$ and $R=\sum_{1\leq i,j \leq n} A_{ij} A_{ij}^*$ and $F,G$ are the block matrices $(A_{ij})_{1\leq i,j \leq n}$, $(A_{ij}^*)_{1\leq i,j \leq n}$ respectively.
\end{theorem}

For Rademacher and Gaussian chaos, higher moments are controlled by lower moments, a property known as ``hypercontractivity'' \cite{arcones93, pena99}. This leads to exponential tail bounds by Markov's inequality as we illustrate below. The same result appears as Proposition 6.5 of \cite{rauhut2010compressive}.

\begin{proposition}\label{thm:highermoments}
Let $X$ be a nonnegative random variable. Let $\sigma,c,\alpha>0$. Suppose $(\E X^s)^{1/s}\leq \sigma c^{1/s} s^{1/\alpha}$ for all $s_0 \leq s < \infty$. Then for any $k>0$ and $u\geq s_0^{1/\alpha}$,
$$\P{X \geq e^{k}\sigma u}\leq c \exp(-k u^{\alpha}).$$
\end{proposition}
\begin{proof}
By Markov's inequality, for any $s>0$, $\P{X \geq e^k \sigma u}\leq\frac{\E X^s}{(e^k \sigma u)^s}\leq c \left( \frac{\sigma s^{1/\alpha}}{e^k \sigma u} \right)^{s}$. Pick $s=u^{\alpha}\geq s_0$ to complete the proof.
\end{proof}

\begin{proposition}\label{thm:tailbound}
Let $(u_i)$ be a Rademacher or Gaussian sequence and $C_1,C_2$ be constants obtained from Theorem \ref{thm:decoupleR} or \ref{thm:decoupleG}. Let $\{A_{ij}\}_{1\leq i,j\leq n}$ be a set of $p$ by $p$ matrices, and assume that the diagonal entries $A_{ii}$ are positive semidefinite. Define
 $M=\sum_i u_i u_j A_{ij}$ and $\sigma=C_1 \max(\norm{Q}^{1/2},\norm{R}^{1/2},\norm{F},\norm{G})$ where $Q,R,F,G$ are as defined in Theorem \ref{thm:khintchinetwo}. Then
 $$\P{\norm{M-\E M}\geq e \sigma u}\leq (2 C_2 n p) \exp(-u).$$
\end{proposition}
\begin{proof}
We will prove the Gaussian case first. Recall that the $s$-Schatten and spectral norms are equivalent: for any $A\in \cplexes^{r \times r}$, $\norm{A}\leq \norm{A}_{C_s}\leq r^{1/s} \norm{A}$. Apply the decoupling inequality, that is Theorem \ref{thm:decoupleG}, and deduce that for any $s\geq 2$,
$$
\left(\E \norm{M-N}^s\right)^{1/s} \leq C_1 C_2^{1/s} \left(\E \norm{\sum_{1\leq i,j\leq n} u_i \tilde{u}_j A_{ij}}_{C_s}^s\right)^{1/s}.$$

Invoke Khintchine's inequality, that is Theorem \ref{thm:khintchinetwo},  and obtain
\begin{align*}
\left(\E \norm{M-N}^s\right)^{1/s} &\leq C_1 (2 C_2)^{1/s} s\max(\norm{Q^{1/2}}_{C_s},\norm{R^{1/2}}_{C_s},\norm{F}_{C_s},\norm{G}_{C_s})\\
&\leq C_1 (2 C_2 np)^{1/s} s \max(\norm{Q}^{1/2}_{},\norm{R}^{1/2}_{},\norm{F}_{},\norm{G})\\
&\leq \sigma(2 C_2 n p)^{1/s}s.
\end{align*}

Apply Proposition \ref{thm:highermoments} with $c=2 C_2 n p$ and $k=\alpha=1$ to complete the proof for the Gaussian case. For the Rademacher case, we take similar steps. First, decouple $(\E \norm{M-N}^s)^{1/s}$ using Theorem \ref{thm:decoupleR}. This leaves us a sum that excludes the $A_{ii}$'s. Apply Khintchine's inequality with the $A_{ii}$'s zeroed. Of course, $Q,R,F,G$ in Proposition \ref{thm:khintchinetwo} will not contain any $A_{ii}$'s, but this does not matter because $A_{ii}^* A_{ii}$ and $A_{ii} A_{ii}^*$ and $A_{ii}$ are all positive semidefinite for any $1\leq i\leq n$ and we can add them back. For example, $\norm{(A_{ij})_{1\leq i\neq j \leq n}}\leq \norm{(A_{ij})_{1\leq i,j\leq n}}$ as block matrices.
\end{proof}
\end{appendices}

An alternative way to prove the Gaussian case of Proposition \ref{thm:tailbound} is to split $(\E \norm{M-N}^s)^{1/s}$ into two terms $(\E\norm{\sum_i (u_i^2-1) A_{ii}})^{1/s}$ and $(\E\norm{\sum_i u_i u_j  A_{ij}})^{1/s}$. For the second term, we can insert Rademacher variables, condition on the Gaussians, decouple the Rademacher sum and apply Theorem \ref{thm:khintchinetwo}. After that, we can pull out the maximum of all the Gaussians from $Q,R,F,G$. Nevertheless, as this may introduce extra $\log n$ factors, we prefer to simply appeal to \cite{arcones93} to decouple the Gaussian sum right away.

\bibliography{matprobe}

\begin{thebibliography}{10}

\bibitem{andrieu03}
{\sc C.~Andrieu, N.~De~Freitas, A.~Doucet, and M.I. Jordan}, {\em {An
  introduction to {MCMC} for machine learning}}, Machine learning, 50 (2003),
  pp.~5--43.

\bibitem{arcones93}
{\sc M.A. Arcones and E.~Gin{\'e}}, {\em {On decoupling, series expansions, and
  tail behavior of chaos processes}}, Journal of Theoretical Probability, 6
  (1993), pp.~101--122.

\bibitem{buchholz2001operator}
{\sc A.~Buchholz}, {\em Operator khintchine inequality in non-commutative
  probability}, Mathematische Annalen, 319 (2001), pp.~1--16.

\bibitem{candes2006stable}
{\sc E.J. Candes, J.K. Romberg, and T.~Tao}, {\em Stable signal recovery from
  incomplete and inaccurate measurements}, Communications on pure and applied
  mathematics, 59 (2006), pp.~1207--1223.

\bibitem{chan1992interface}
{\sc T.F.C. Chan and T.P. Mathew}, {\em The interface probing technique in
  domain decomposition}, SIAM Journal on Matrix Analysis and Applications, 13
  (1992), p.~212.

\bibitem{chan1985survey}
{\sc T.F. Chan and D.C. Resasco}, {\em A survey of preconditioners for domain
  decomposition.}, tech. report, DTIC Document, 1985.

\bibitem{chang00}
{\sc E.C. Chang, S.~Mallat, and C.~Yap}, {\em Wavelet foveation}, Applied and
  Computational Harmonic Analysis, 9 (2000), pp.~312--335.

\bibitem{pena99}
{\sc V.~De~la Pe{\~n}a and E.~Gin{\'e}}, {\em {Decoupling: from dependence to
  independence}}, Springer Verlag, 1999.

\bibitem{demanet2011matrix}
{\sc L.~Demanet, P.D. L{\'e}tourneau, N.~Boumal, H.~Calandra, J.~Chiu, and
  S.~Snelson}, {\em Matrix probing: a randomized preconditioner for the
  wave-equation {H}essian}, Applied and Computational Harmonic Analysis,
  (2011).

\bibitem{demanet2010discrete}
{\sc L.~Demanet and L.~Ying}, {\em Discrete symbol calculus}, SIAM Review, 53
  (2011), pp.~71--104.

\bibitem{edelman1989eigenvalues}
{\sc A.~Edelman}, {\em {Eigenvalues and condition numbers of random matrices}},
  PhD thesis, Massachusetts Institute of Technology, 1989.

\bibitem{folland95}
{\sc G.B. Folland}, {\em {Introduction to partial differential equations}},
  Princeton Univ Pr, 1995.

\bibitem{hale:2505}
{\sc N.~Hale, N.J. Higham, L.N. Trefethen, et~al.}, {\em Computing a$\alpha$,
  log (a), and related matrix functions by contour integrals}, SIAM Journal on
  Numerical Analysis, 46 (2008), pp.~2505--2523.

\bibitem{halko}
{\sc N.~Halko, P.~Martinsson, and J.~Tropp}, {\em Finding structure with
  randomness: Probabilistic algorithms for constructing approximate matrix
  decompositions}, SIAM Review, 53 (2011), p.~217.

\bibitem{higham2008functions}
{\sc N.J. Higham}, {\em {Functions of matrices: theory and computation}},
  Society for Industrial Mathematics, 2008.

\bibitem{karger96}
{\sc D.R. Karger and C.~Stein}, {\em {A new approach to the minimum cut
  problem}}, Journal of the ACM, 43 (1996), pp.~601--640.

\bibitem{karp1987efficient}
{\sc R.M. Karp and M.O. Rabin}, {\em Efficient randomized pattern-matching
  algorithms}, IBM Journal of Research and Development, 31 (1987),
  pp.~249--260.

\bibitem{lust86}
{\sc F.~Lust-Piquard}, {\em {In{\'e}galit{\'e}s de {K}hintchine dans {C}p}}, CR
  Acad. Sci. Paris, 303 (1986), pp.~289--292.

\bibitem{lust91}
{\sc F.~Lust-Piquard and G.~Pisier}, {\em {Non-commutative {K}hintchine and
  {P}aley inequalities}}, Arkiv f{\"o}r Matematik, 29 (1991), pp.~241--260.

\bibitem{pfander2010sparsity}
{\sc G.E. Pfander and H.~Rauhut}, {\em Sparsity in time-frequency
  representations}, Journal of Fourier Analysis and Applications, 16 (2010),
  pp.~233--260.

\bibitem{pfander2008identification}
{\sc G.E. Pfander, H.~Rauhut, and J.~Tanner}, {\em Identification of matrices
  having a sparse representation}, Signal Processing, IEEE Transactions on, 56
  (2008), pp.~5376--5388.

\bibitem{pfander2011restricted}
{\sc G.E. Pfander, H.~Rauhut, and J.A. Tropp}, {\em The restricted isometry
  property for time-frequency structured random matrices}, Arxiv preprint
  arXiv:1106.3184,  (2011).

\bibitem{rauhut09}
{\sc H.~Rauhut}, {\em {Circulant and {T}oeplitz matrices in compressed
  sensing}}, Proc. SPARS, 9 (2009).

\bibitem{rauhut2010compressive}
\leavevmode\vrule height 2pt depth -1.6pt width 23pt, {\em Compressive sensing
  and structured random matrices}, Theoretical Foundations and Numerical
  Methods for Sparse Recovery, 9 (2010), pp.~1--92.

\bibitem{shubin01}
{\sc M.A. Shubin}, {\em {Pseudodifferential operators and spectral theory}},
  Springer Verlag, 2001.

\bibitem{tropp2008conditioning}
{\sc J.A. Tropp}, {\em On the conditioning of random subdictionaries}, Applied
  and Computational Harmonic Analysis, 25 (2008), pp.~1--24.

\end{thebibliography}
\bibliographystyle{siam}

\end{document}